\documentclass[12pt]{article}
\usepackage{amsthm}

\usepackage[english]{babel}
\usepackage{fancyhdr}

\usepackage[curve,matrix,arrow,cmtip]{xy}
\NoComputerModernTips

\newdir^{ (}{{}*!/-3pt/\dir^{(}}    
\newdir_{ (}{{}*!/-3pt/\dir_{(}}

 \usepackage[ansinew]{inputenc}
\usepackage[dvipsnames,usenames]{color} 

\usepackage[pdftex,%
pdfpagemode={UseOutlines},%
bookmarks,bookmarksopen,%
pdfstartview={FitH},%
linkcolor={RawSienna},citecolor={OliveGreen},%
menucolor=Yellow, urlcolor={red},colorlinks=true,
hyperindex]%
{hyperref}

\usepackage[scale=0.7,hmarginratio=1:1,vmarginratio=2:3]{geometry}

\newcounter{themargin}
\long\def\forget#1{}

\theoremstyle{plain} 
\newtheorem{Thm}{Theorem}[section]

\newtheorem{Prop}[Thm]{Proposition}

\newtheorem{Lem}[Thm]{Lemma}

\newtheorem{Cor}[Thm]{Corollary}

\newtheorem{thm}[Thm]{Theorem}

\newtheorem{defi}[Thm]{Definition}
\newtheorem{rema}[Thm]{Remark} 
\newtheorem{prop}[Thm]{Proposition}
\newtheorem{lemm}[Thm]{Lemma}
\newtheorem{exam}[Thm]{Example}

\newtheorem{coro}[Thm]{Corollary}

\theoremstyle{definition} 
\newtheorem{Def}[Thm]{Definition}

\newcommand{\plim}{\mathop{\varprojlim}\limits}

\newcommand{\OO}{{\cal O}}

\newcommand{\trdeg}{\mathrm{trdeg}}
\newcommand{\mf}{\mathfrak}
\newcommand{\ol}[1]{\overline{#1}}

\newcommand{\im}{\mathrm{im}}

\newcommand{\Jor}{\mathrm{Jor}}

\def\Ff{{\mathbb{F}}}
\def\Rr{{\mathbb{R}}}
\def\Cc{{\mathbb{C}}}
\def\Qq{{\mathbb{Q}}}
\def\Zz{{\mathbb{Z}}}
\def\Aa{{\mathbb{A}}}
\def\Nn{{\mathbb{N}}}
\def\Gg{{\mathbb{G}}}
\def\Ll{{\mathbb{L}}}

\theoremstyle{remark} 
\newtheorem{Rem}[Thm]{Remark}

\forget{\newif\ifnormalesBeweisEnde

  {\vskip 0.3ex plus 0.5ex minus 0ex \pagebreak[1]
   \global\normalesBeweisEndetrue
   \trivlist
   \item[\hskip\labelsep {\textsc{Proof} \rm #1}:]}%
  {\ifnormalesBeweisEnde \EndOfProof \fi
   \endtrivlist
   \vskip 1ex plus 1ex minus 0ex \pagebreak[2]}
  {\vskip 0.3ex plus 0.5ex minus 0ex \pagebreak[1]
   \global\normalesBeweisEndetrue
   \trivlist
   \item[\hskip\labelsep \textsc{Proof}:]}%
  {\ifnormalesBeweisEnde \EndOfProof \fi
   \endtrivlist
   \vskip 1ex plus 1ex minus 0ex \pagebreak[2]}
\def\EndOfProof{\hskip .5em \vrule width 1.0ex height 1.0ex depth 0.3ex}
}

\usepackage{latexsym}
\usepackage{amsmath}
\usepackage{amstext}
\usepackage{amssymb}
\usepackage{amsbsy}

\font\tencyr=wncyr10
\font\sevencyr=wncyr7
\font\fivecyr=wncyr5

\newfam\cyrfam
\textfont\cyrfam\tencyr
\scriptfont\cyrfam\sevencyr
\scriptscriptfont\cyrfam\fivecyr

\def\theenumi{(\alph{enumi})}

\def\p@enumii{\theenumi}

\newcommand{\Prime}{\kern3\fontdimen1\font$'$\kern-7\fontdimen1\font}

\newdimen\auxdimen
\newdimen\auxdim
\newdimen\auxdimone
\newdimen\auxdimtwo
\newdimen\auxdimthree

\long\def\forget#1{}

\def\?{\ ???\ \immediate\write16{}%
\immediate\write16{Warning: There was still a question mark . . . }%
\immediate\write16{}}

\newcommand{\BA}{{\mathbb{A}}}

\newcommand{\BF}{{\mathbb{F}}\,\!{}}

\newcommand{\BL}{{\mathbb{L}}}

\newcommand{\BN}{{\mathbb{N}}}

\newcommand{\BQ}{{\mathbb{Q}}}

\newcommand{\BZ}{{\mathbb{Z}}}

\newcommand{\CA}{{\cal A}}

\newcommand{\CF}{{\cal F}}

\newcommand{\CN}{{\cal N}}

\newcommand{\CR}{{\cal R}}
\newcommand{\CS}{{\cal S}}

\newcommand{\CV}{{\cal V}}

\newcommand{\CX}{{\cal X}}

\newcommand{\wt}[1]{\widetilde{#1}}

\DeclareMathOperator{\Trace}{Trace}

\newcommand{\image}{\mathop{{\rm Im}}\nolimits}

\newcommand{\kernel}{\mathop{\rm Ker}\nolimits}

\def\invlim{\mathop{\vtop{\hbox{\rm lim}\vskip-8pt
        \hbox{\hskip1pt$\scriptstyle\longleftarrow$}\vskip-1pt}}}

\newcommand{\Hom}{\mathop{\rm Hom}\nolimits}

\newcommand{\Spec}{\mathop{{\rm Spec}}\nolimits}

\newcommand{\GL}{{\rm GL}}

\newcommand{\SL}{{\rm SL}}

\DeclareMathOperator{\Char}{char}

\newcommand{\Gal}{\mathop{\rm Gal}\nolimits}

\DeclareMathOperator{\Aut}{Aut}

\newcommand{\red}{{\rm red}}

\def\phi{\varphi}
\def\setminus{\smallsetminus}

\let\emptyset\varnothing

\def\longto{\longrightarrow}
\def\into{\hookrightarrow}

\newbox\mybox

\def\ontoover#1{\mathrel{
       \setbox\mybox=\hbox spread 1.4em{\hfil$\scriptstyle#1$\hfil}
       \vbox{\offinterlineskip\copy\mybox
             \hbox to\wd\mybox{\rightarrowfill\hskip-2.8mm
                               $\rightarrow$}}}}
\def\onto{\ontoover{\ }}

\let\oldbullet\bullet
\def\bullet{{\mathchoice{\oldbullet}%
                        {\oldbullet}%
                        {\scriptscriptstyle\oldbullet}%
                        {\oldbullet}}}

\let\emptyset\varnothing

\newcommand\ssi{{\rm ss}}

\DeclareMathOperator{\codim}{codim}

\newcommand{\resg}{|_}

\newcommand{\sss}{\mathrm{sss}}

\title{\vspace{-1em}
Independence of $\ell$-adic representations\\
of geometric Galois groups}
\author{\large G.\ B\"ockle, W.\ Gajda and S.\ Petersen}
\date{\today}

\begin{document}
\maketitle
\begin{abstract} 
Let $k$ be an algebraically closed field of arbitrary characteristic, let
$K/k$ be a finitely generated field extension and let $X$ be a separated
scheme of finite type over $K$. For each prime $\ell$, the absolute Galois
group of $K$ acts on the $\ell$-adic etale cohomology modules of $X$. We prove
that this family of representations varying over $\ell$ is almost
independent in the sense of Serre, i.e., that the fixed fields inside an
algebraic closure of $K$ of the kernels of the representations for all $\ell$
become linearly disjoint over a finite extension of $K$. In doing this, we
also prove a number of interesting facts on the images and on the ramification
of this family of representations. 
\end{abstract}


\section{Introduction}
Let $G$ be a profinite group and $L$ a set of prime numbers. For every $\ell\in L$ let $G_\ell$ be a profinite
group and $\rho_\ell: G\to G_\ell$ a homomorphism. Denote by
$$\rho: G\to \prod\limits_{\ell\in L} G_\ell$$
the homomorphism induced by the $\rho_\ell$. Following the notation in \cite{bible} 
we call the family $(\rho_\ell)_{\ell\in L}$ {\em independent}
if $\rho(G)=\prod\limits_{\ell\in L} \rho_\ell(G)$. The family $(\rho_\ell)_{\ell\in L}$ is said to be {\em almost independent} if there exists an open subgroup $H$ of $G$ such that $\rho(H)=\prod_{\ell\in L} \rho_\ell(H)$. 

The main examples of such families of homomorphisms arise as follows: Let $K$ be a field  {with algebraic closure $\wt K$ and absolute Galois group $\Gal(K)=\Aut(\wt K/K)$. Let $X/K$ be a separated algebraic scheme\footnote{A scheme $X/K$ is algebraic if the structure morphism $X\to\Spec K$ is of finite type (cf.~\cite[Def.~6.4.1]{EGAI}).}
and denote by $\Ll$ the set of all prime numbers}. For every $q\in \Nn$ 
and every $\ell\in \Ll\smallsetminus \{\Char(k)\}$ we shall consider the representations
$$\rho_{\ell, X}^{(q)}: \Gal(K)\to \Aut_{\Qq_\ell}(H^q(X_{\wt K}, \Qq_\ell))\quad \mbox{and} \quad
\rho_{\ell, X, c}^{(q)}: \Gal(K)\to \Aut_{\Qq_\ell}(H^q_c(X_{\wt K}, \Qq_\ell))$$ 
of {$\Gal(K)$} on the \'etale cohomology groups $H^q(X_{\wt K}, \Qq_\ell)$ and $H_c^q(X_{\wt K}, \Qq_\ell)$. The following independence result has recently been obtained.

\begin{thm}\label{Serre-Gajda-Petersen}
Let $K$ be a finitely generated extension of~$\,\Qq$ and let $X/K$ be a separated algebraic scheme. Then the families $(\rho_{\ell, X}^{(q)})_{\ell\in \Ll}$ and  $(\rho_{\ell, X, c}^{(q)})_{\ell\in \Ll}$ are almost independent.
\end{thm}

The proof of this statement in the important special case $\trdeg(K/\Qq)=0$ is due to Serre (cf.~\cite{bible}). The case 
$\trdeg(K/\Qq)>0$ was worked out in \cite{gp}, answering a question of Serre (cf.~\cite{bible}, \cite{serre1994})
and Illusie \cite{illusie}.

The usefulness of almost independence is alluded to in Serre
\cite[Introd.]{bible} (cf.~also \cite[Sect.~10]{serre1994}). Almost independence for a family 
$(\rho_\ell\colon \Gal(K)\to G_\ell)_{\ell\in L}$ means that after a
finite field extension $E/K$, the image of $\Gal(E)$ under the
product representation $\prod_{\ell\in L}\rho_\ell$ is the product
$P=\prod_{\ell\in L}\rho_{\ell}(\Gal(E))$ of the images. In particular for any
finite extension $F/E$, the image of $\Gal(F)$ is open in $P$. This has
applications if one has precise knowledge of the shape of the images for
all $\ell$. For instance, suppose that there exists a reductive algebraic
subgroup $G$ of some $\GL_n$ over $\BQ$ such that for all sufficiently
large finite extensions $F$ of $K$ the image $\rho_\ell(\Gal(F))$ is
open in $G(\BQ_\ell)\cap \GL_n(\BZ_\ell)$ for all $\ell$ and surjective
for almost all $\ell$. Then almost independence implies that for some
finite extension $E$ of $K$ the image $\Gal(E)$ is adelically open,
i.e., it is open in the restricted product $\prod_{\ell\in L}^\prime
G(\BQ_\ell)$. For $K$ a number field, the Mumford-Tate conjecture
(cf.~\cite[C.3.3, p.~387]{serre1972}) predicts a group $G$ as above if $\rho_\ell$ arises from an
abelian variety over~$K$.

{The present article is concerned with a natural variant of Theorem~\ref{Serre-Gajda-Petersen} that grew out of the study of independence of families over fields of positive characteristic. For $K$ a finitely generated extension of $\BF_p$ it has long been known, e.g.~\cite{Igusa} 
or \cite{devicpink}, that the direct analogue of Theorem~\ref{Serre-Gajda-Petersen} is false: If $\varepsilon_\ell: \Gal(\BF_p)\to \Zz_\ell^\times$ denotes the $\ell$-adic cyclotomic character that describes the Galois action on $\ell$-power roots of unity, then it is elementary to see that the family $(\varepsilon_\ell)_{\ell\in\BL\setminus\{ p\}}$  is not almost independent. It follows from this that for every 
abelian variety $A/K$, if we denote by $\sigma_{\ell, A}: \Gal(K)\to \Aut_{\Qq_\ell}(T_\ell(A))$ the representation of $\Gal(K)$ on the  $\ell$-adic Tate module of $A$, then $(\sigma_{\ell, A})_{\ell\in\Ll\smallsetminus \{ p\}}$ is {\em not} almost  independent. One is thus led to study independence over the compositum $\wt\Ff_p K$ obtained from the field $K$ by adjoining all roots of unity. Having gone that far, it is then natural to study independence over any field $K$ that is finitely generated over an arbitrary algebraically closed field~$k$. Our main result is the following independence theorem.} 

\begin{thm} \label{main} {\bf (cf.~Theorem \ref{mainmain})} Let $k$ be an algebraically closed field of characteristic $p\ge 0$. 
Let $K/k$ be a finitely generated extension and let $X/K$ be a separated 
algebraic scheme. Then the families $(\rho_{\ell, X}^{(q)}{\resg{\Gal(K)}})_{\ell\in \Ll\smallsetminus\{p\}}$ and 
$(\rho_{\ell, X, c}^{(q)}{\resg{\Gal(K)}}))_{\ell\in \Ll\smallsetminus \{p\}}$ are almost independent.
\end{thm}

It will be clear that many techniques of the present article rely on \cite{bible}. Also, some of the key results of \cite{gp} will be important. The new methods in comparison with the previous results are the following: (i) The analysis of the target of our Galois representations, reductive algebraic groups over $\BQ_\ell$, will be based on a structural result by Larsen and Pink (cf.~\cite{lp2011}) and no longer as for instance in \cite{bible} on extensions of results by Nori (cf.~\cite{nori1987}). In the technically useful case $k\neq \wt k$, this facilitates greatly the passage from $\Gal(K)$ to $\Gal(K\wt k)$ when studying their image under $\rho_{\ell,X,?}^{(q)}$. (ii) Since we also deal with cases of positive characteristic, ramification properties will play a crucial role to obtain necessary finiteness properties of fundamental groups. The results on alterations by de Jong (cf.~\cite{dejong}) will obviously be needed. However we were unable to deduce all needed results from there, despite some known semistability results that follow from~\cite{dejong}. Instead we carry out a reduction to the case where $K$ is absolutely finitely generated and where $X/K$ is smooth and projective (this uses again \cite{dejong}). (iii) In the latter case,  we use a result by Kerz-Schmidt-Wiesend (cf.~\cite{kerzschmidt}) that allows one to control ramification on $X$ by controlling it on all smooth curves on~$X$. Since $X$ is smooth, results of Deligne show that the semisimplifications of $\rho_{\ell,X,?}^{(q)}$ form a pure and strictly compatible system. On curves, we can then apply the global Langlands correspondence proved by Lafforgue in \cite{Lafforgue}. This is a deep tool, but it allows us to obtain a very clean conclusion about the ramification properties of $(\rho_{\ell,X,?}^{(q)})_{\ell\in \Ll\smallsetminus \{p\}}$. 

Part (i) is carried out in Section~\ref{groupconcepts}. Results on fundamental groups and first results on ramification are the theme of Section~\ref{FundGps}; this includes parts of (ii) and we also refine some results from \cite{kerzschmidt}. Section~\ref{indep-sec} provides the basic independence criterion on which our proof of Theorem~\ref{main} ultimately rests. Section~\ref{Reductions} performs the reductions mentioned in (ii). The ideas described in (iii) are concluded in Section~\ref{TheProof}, where a slightly more precise form of Theorem~\ref{main} is proved.

We would like to point out that an alternative method for the part (ii) of our approach could be based on a recent unpublished result by Orgogozo which proves a global semistable reduction theorem (cf. \cite[2.5.8. Prop.]{orgo}). 
When our paper was complete we were informed by Anna Cadoret that, together with Akio Tamagawa, she has proven our Theorem \ref{main} by a different method cf. \cite{cata}. 


{\bf Acknowledgments:} G.B. thanks the Fields Institute for a research stay in the spring of 2012 during which part of this work was written. He also thanks Adam Mickiewicz University in Pozna{\' n} for making possible a joint visit of the three authors in the fall of 2012. He is supported by a grant of the DFG within the SPP 1489. W.G. thanks the Interdisciplinary Center for Scientific Computing (IWR) at Heidelberg University for hospitality during a research visit in January 2012 shortly after this project had been started. He was partially supported by the Alexander von Humboldt Foundation and by research grant 
UMO-2012/07/B/ST1/03541 of the National Centre of Sciences of Poland. S.P. thanks the Mathematics Department at Adam Mickiewicz University for hospitality and support during several research visits. We thank F. Orgogozo and L. Illusie for interesting correspondence concerning this project.

\section{Notation}\label{notation}
For a field $K$ with algebraic closure $\wt K$, we denote by $K_s\subset \wt K$ a separable closure. Then $\Gal(K)$ is equivalently defined as $\Gal(K_s/K)$ and as $\Aut(\wt K/K)$, since any field automorphism of $K_s$ fixing $K$ has a unique extension to $\wt K$. If $E/K$ is an arbitrary field extension, and if $\wt K$ is chosen inside $\wt E$, then there is a natural isomorphism $\Aut(\wt K/\wt K\cap E)\stackrel\simeq\longto \Aut(\wt K E/ E)$. Composing its inverse with the natural restriction $\Gal(E)\to\Aut(E\wt K/E)$ one obtains a canonical map which we denote $\mathrm{res}_{E/K}:  \Gal(E)\to \Gal(K)$. If $E/K$ is algebraic, then $\mathrm{res}_{E/K}$ is injective and we identify $\Gal(E)$ with the subgroup $\mathrm{res}_{E/K}(\Gal(E))=\Gal(E\cap \wt K)$ of $\Gal(K)$. 

Let $G$ be a profinite group. A {\em normal series} in $G$ is a sequence 
$$G\triangleright N_1\triangleright N_2\triangleright \cdots \triangleright N_s=\{e\}$$
of closed subgroups such that each $N_i$ is normal in $G$.

A $K$-variety $X$ is a scheme $X$ that is integral separated and algebraic over $K$. We denote by $K(X)$ its function field. Let $S$ be a normal connected scheme with function field $K$. A separable algebraic extension $E/K$ is said to be {\em unramified along $S$} if for every finite extension $F/K$ inside $E$ the normalization of $S$ in $F$ is \'etale over $S$. We usually consider $S$ as a scheme equipped with the generic geometric base point $s: \Spec(\widetilde{K})\to S$ and denote by $\pi_1(S):=\pi_1(S,  s)$ the \'etale fundamental group of $S$. If $\Omega$ denotes the maximal extension of $K$ in $K_s$ which is unramified along $S$, then $\pi_1(S)$ can be identified with the Galois group $\Gal(\Omega/K)$. A homomorphism $\rho: \Gal(K)\to H$ is said to be {\em unramified along $S$} if the fixed field $K_s^{\ker(\rho)}$ is unramified along $S$. If $E/K$ is an arbitrary algebraic extension, then $\rho\resg{\Gal(E)}$ stands for $\rho\circ \mathrm{res}_{E/K}$.

\section{Concepts from group theory}\label{groupconcepts}
In this section, we prove a structural result for compact profinite subgroups of linear algebraic groups over $\wt\Qq_\ell$ (cf.~Theorem~\ref{lp-main}) that will be crucial for the proof of the main theorem of this article. It is a consequence of  a variant (cf.~Proposition~\ref{lp0}) of a theorem of Larsen and Pink (cf.~\cite[Thm.~0.2, p.~1106]{lp2011}). The proof of~Proposition~\ref{lp0} makes strong use of the results and methods in~\cite{lp2011}, and in particular does not depend on the classification of finite simple groups.
\smallskip

\begin{defi}  \label{sigmadef} 
For $c\in \Nn$ we denote by $\Sigma_\ell(c)$ the class of profinite groups $M$ which possess a normal series by open subgroups 
$$M\triangleright I\triangleright P\triangleright \{1\}$$ such that $M/I$ is a finite product of finite simple groups of 
Lie type in characteristic $\ell$, the group $I/P$ is finite abelian of order prime to $\ell$ and index $[I:P]\le c$, and $P$ is a
pro-$\ell$ group.
\end{defi}

\begin{defi}  \label{jordef} 
For $d\in \Nn$ and $\ell$ a prime we denote by $\Jor_\ell(d)$ the class of finite groups $H$ which possess a normal abelian subgroup $N$ of order prime to $\ell$ and of index $[H:N]\le d$. We define $\Jor(d)$ as the union of the $\Jor_\ell(d)$ over all primes $\ell$. 
\end{defi}

\begin{defi}  \label{nbddgp-def}
A profinite group $G$ is called {\em $n$-bounded at $\ell$} if there exist closed compact subgroups $G_1\subset G_2\subset \GL_n(\wt\BQ_\ell)$ such that $G_1$ is normal in $G_2$ and $G\cong G_1/G_2$.
\end{defi}

The following is the main result of this section. 
\begin{thm}\label{lp-main}
For every $n\in \BN$ there exists a constant $J'(n)$ (independent of $\ell$) such that the following holds: For any prime $\ell$, any group $G$ that is $n$-bounded at $\ell$ lies in 
a short exact sequence
$$1\to M\to G \to H \to 1$$
such that $M$ is open normal in $G$ and lies in $\Sigma_\ell(2^n)$ and $H$ lies in $\Jor_\ell(J'(n))$. 
\end{thm}
We state an immediate corollary:
\begin{coro}\label{lp-cor}
Let $G$ be $n$-bounded at $\ell$ and define $G_\ell^+$ as the normal hull of all pro-$\ell$ Sylow subgroups of $G$. Then for $\ell>J'(n)$, the group $G_\ell^+$ is an open normal subgroup of $M$ of index at most~$2^n$.
\end{coro}

In the remainder of this section we shall give a proof of Theorem~\ref{lp-main}. Moreover we shall derive some elementary permanence properties for the properties described by~$\Sigma_\ell(d)$ and~$\Jor_\ell(d)$. 

\smallskip

The content of the following lemma is presumably well-known. 
\begin{lemm} For every $r\in \Nn$, every algebraically closed
field $F$ and every semisimple 
algebraic group $G$ of rank $r$ the center $Z$ of $G$ satisfies
$|Z(F)|\le 2^r$.\label{centfin}
\end{lemm}

\begin{proof}
Lacking a precise reference, we include a proof for the reader's convenience. Observe first that  the center $Z$ is a finite (cf.~\cite[I.6.20, p. 43]{testerman}) diagonalizable algebraic group. Let $T$ be a maximal torus of $G$. Denote by $X(T)=\Hom(T, \Gg_m)$ the character group of $T$ and by $\Phi\subset X(T)$ the set of roots of $G$. Then ${\mathcal{R}}=(X(T)\otimes \Rr, \Phi)$ is a root system. Let $P=\Zz \Phi$ be the root lattice and $Q$ the weight lattice of this root system. Then $P\subset X(T)\subset Q$. The center $Z$ of $G$ is the kernel of the adjoint representation (cf.~\cite[I.7.12, p. 49]{testerman}). Hence $Z=\bigcap_{\chi\in \Phi} \ker(\chi)$ and there is an exact sequence
$$0\to Z\to T\to \prod_{\chi\in \Phi} \Gg_m$$
where the right hand map is induced by the characters $\chi: T\to \Gg_m$  ($\chi\in \Phi$). We apply the functor $\Hom(-, \Gg_m)$ and obtain an exact sequence 
$$\prod_{\chi\in \Phi} \Zz\to X(T)\to \Hom(Z, \Gg_m)\to 0$$
The cokernel of the left hand map is $X(T)/P$. Thus $|Z(F)|\le [X(T):P]\le [Q:P]$. 

Furthermore, the root system ${\mathcal{R}}$ decomposes into a direct sum
$${\mathcal{R}}=\bigoplus_{i=1}^s (E_i, \Phi_i)$$
of indecomposable root systems ${\mathcal{R}}_i:=(E_i, \Phi_i)$. Let $r_i=\dim(E_i)$ be the rank of ${\mathcal{R}}_i$. Let $P_i$ be the root lattice and $Q_i$ the weight lattice of ${\mathcal{R}}_i$. Note that by definition $P=\oplus_iP_i$ and $Q=\oplus_iQ_i$. It follows from the classification of indecomposable root systems that $|Q_i/P_i|\le 2^{r_i}$ (cf.~\cite[Table 9.2, p. 72]{testerman}) for all $i$. Hence $|Z(F)|\le|Q/P|\le 2^{r_1}2^{r_2}\cdots 2^{r_s}=2^r$ as desired. 
\end{proof}

\begin{rema}
The semisimple algebraic group $(\SL_{2, \Cc})^r$ has rank $r$ and its center $(\mu_2)^r$ has exactly 
$2^r$ $\Cc$-rational points. Hence the bound of Lemma \ref{centfin} cannot be improved.
\end{rema}

The following result is an adaption of the main result of \cite{lp2011} by Larsen and Pink.
\begin{prop} \label{LPProp}For every $n\in\Nn$, there exists a constant $J'(n)$ such that for every field $F$
of positive characteristic $\ell$ and every finite subgroup $\Gamma$ of 
$GL_{n}(F)$, there exists a normal series 
$$\Gamma\triangleright L\triangleright M\triangleright I\triangleright P\triangleright \{1\}$$
of $\Gamma$ with the following properties:
\begin{enumerate}
\item[i)] $[\Gamma:L]\le J'(n)$.
\item[ii)] The group $L/M$ is abelian of order prime to $\ell$.
\item[iii)] The group $M/I$ is a finite 
product of finite simple groups of Lie
type in characteristic $\ell$.
\item[iv)] The group $I/P$ is abelian of order prime to $\ell$ and $[I:P]\le 2^n$.
\item[v)] $P$ is an $\ell$-group.
\end{enumerate}\label{lp0}
Furthermore the constant $J'(n)$ is the same as in \cite[Thm.~0.2, p. 1106]{lp2011}. 
\end{prop}

{\em Proof.} We can assume that $F$ is algebraically closed. Let $J'(n)$ be the constant from
\cite[Thm.~0.2, p. 1106]{lp2011}.
Larsen and Pink construct in the proof of their Theorem \cite[Thm.~0.2, p. 1155--1156]{lp2011}
a smooth algebraic group $G$ over $F$ containing $\Gamma$ and normal 
subgroups $\Gamma_i$ of $\Gamma$ such that there is a normal series
$$\Gamma\triangleright\Gamma_1\triangleright \Gamma_2 \triangleright \Gamma_3\triangleright\{1\}$$
and such that $[\Gamma:\Gamma_1]\le J'(n)$, $\Gamma_1/\Gamma_2$ is a product of finite simple groups of Lie type
in characteristic $\ell$, $\Gamma_2/\Gamma_3$ is abelian of order prime to $\ell$ and $\Gamma_3$ is an $\ell$-group.
Let $R$ be the unipotent radical of the connected component $G^\circ$ of $G$.
The proof of Larsen and Pink shows that $\Gamma_1\triangleleft G^\circ(F)$, $\Gamma_3=\Gamma\cap R(F)$ and $\Gamma_2/\Gamma_3$ is contained in $\ol{Z}(F)$ where $\ol{Z}$ denotes the center of the reductive group $\ol{G}:=G^\circ/R$. 
Let $\ol{D}=[\ol{G}, \ol{G}]$ be
the derived group of
$\ol{G}$ and $D=[G^\circ, G^\circ]R$.

Now define $L=\Gamma_1$, $M=\Gamma_1\cap D(F)$, $I=\Gamma_2\cap D(F)$ and $P=\Gamma_3$. These groups are normal in $\Gamma$, because $D(F)$ is characteristic in $G^\circ(F)$ and because $\Gamma_1, \Gamma_2, \Gamma_3$ are normal in
$\Gamma$. The group $L/M$ is a subgroup of the abelian group $G^\circ(F)/{D(F)}$. 
The group $M/I$ is a normal subgroup of $\Gamma_1/\Gamma_2$, hence it is a product of finite simple groups of Lie type in characteristic $\ell$. The group $I/P$ is a subgroup of $\Gamma_2/\Gamma_3$, hence 
$I/P$ is abelian of order prime to $\ell$. Furthermore $I/P=I/\Gamma_3$ is a subgroup of $\ol{G}(F)$ which lies in $\ol{D}(F)$ and
in $\ol{Z}(F)$. Thus $I/P$ lies in the center $\ol{Z}(F)\cap \ol{D}(F)$ of the semisimple group $\ol{D}(F)$. 
It follows by Lemma \ref{centfin} that $[I:P]\le 2^{\mathrm{rk}(\ol{D})}$. 

It remains to show that $\mathrm{rk}(\ol{D})\le n$. Let $T$ be a maximal torus of $\ol{D}$ and denote by 
$\pi: G^\circ \to \ol{G}$ the canonical projection. Then the algebraic group $B:=\pi^{-1}(T)$ sits in an exact 
sequence 
$$0\to R\to B\to T\to 0$$
and $B$ is connected smooth and solvable, because $R$ and $T$ have these properties. The above exact sequence
splits (cf.~\cite[XVII.5.1]{SGA3}); hence $B$ contains a copy of $T$. This copy is contained in a maximal torus $T'$ of
$\GL_{n, F}$. Thus $n=\dim(T')\ge \dim(T)=\mathrm{rk}(\ol{D})$ as desired.\hfill $\Box$

\begin{proof}[Proof of Theorem~\ref{lp-main}]
Suppose $G$ is $n$-bounded at $\ell$, so that it is a quotient $G_2/G_1$ with $G_i\subset \GL_n(\wt\Qq_\ell)$. By Lemma~\ref{sigmalemm}(a) below, it will suffice to prove the theorem for $G_2$. Thus we assume that $G$ is a compact profinite subgroup of $\GL_n(\wt\BQ_\ell)$. By compactness of $G$ and a Baire category type argument (cf.~\cite[proof of Cor.~5]{dickinson}) the group $G$ is contained in $\GL_n(E)$ for some finite extension $E$ of $\Qq_\ell$. Let $\OO_E$ be the ring of integers of the local field $E$. Again by compactness of $G$ one can then find an $\OO_E$-lattice in $E^n$ that is stable under $G$. Hence we may assume that $G$ is a closed subgroup of $\GL_n(\OO_E)$.

Let $\mf p$ be the maximal ideal of the local ring $\OO_E$ and let $\Ff=\OO_E/\mf p$ be its residue field. The kernel $\mathcal{K}$ of the canonical map $p: \GL_n(\OO_E)\to \GL_n(\Ff)$ is a pro-$\ell$ group. Hence $Q=\mathcal{K}\cap G$ is pro-$\ell$ and open normal in $G$. We now apply Proposition \ref{lp0} to the subgroup $G/Q$ of $\GL_n(\Ff)\subset \GL_n(F)$ with $F=\overline{\BF}\cong\overline{\BF}_\ell$. This yields a normal series 
$$G\triangleright L\triangleright M\triangleright I\triangleright P\triangleright Q\triangleright \{1\}$$
such that the group $G/M$ lies in $\Jor_\ell(J'(n))$, and the group $M$ lies in $\Sigma_\ell(2^n)$ -- for the latter use that $Q$ is pro-$\ell$ and normal in $G$ and $P/Q$ is a finite $\ell$-group.
\end{proof}

The following lemma records a useful permanence property of groups in~$\Sigma_\ell(c)$ and~$\Jor_\ell(d)$.
\begin{lemm} \label{sigmalemm} 
Fix any $e\in\BN$. Then for any prime number $\ell$ the following holds:
\begin{enumerate}
\item If $H'\unlhd H$ is a normal subgroup of some $H\in\Jor_\ell(e)$, then $H'$ and $H/H'$ lie in $\Jor_\ell(e)$.
\item If $M'\unlhd M$ is a closed normal subgroup of some $M\in\Sigma_\ell(e)$, then $M'$ and $M/M'$ lie in~$\Sigma_\ell(e)$.
\end{enumerate}
\end{lemm}
If $M'$ in part (b) of the lemma was not normal in $M,$ then clearly $M'$ need not lie in $\Sigma_\ell(c)$ again.

\begin{proof} 
We only give the proof of (b), the proof of (a) being similar but simpler. Let $M$ be in $\Sigma_\ell(e)$ and consider a normal series $M\triangleright I\triangleright P\triangleright \{1\}$ as in Definition \ref{sigmadef}. Then $L:=M/I$ is isomorphic to a product $L_1\times \cdots \times L_s$ for certain finite simple groups of Lie type $L_i$ in characteristic $\ell$. Suppose $M'$ is a closed normal subgroup of $M$ and define $\ol{M'}=M'I/I$. By Goursat's Lemma the groups $\ol{M'}$ and $M'/\ol{M'}$ are products of some of the $L_i$. From this it is straightforward to see that both $M'$ and $M/M'$ lie in $\Sigma_\ell(c)$.
\end{proof}

The following corollary is immediate from Lemma~\ref{sigmalemm}(b):
\begin{coro}\label{sigmal-cor}
Fix a constant $c\in\Nn$. Let $G$ be a profinite group, and for each $\ell\in L$ let $\rho_\ell\colon G\to G_\ell$ be a homomorphism of profinite groups such that $\image(\rho_\ell)\in\Sigma_\ell(c)$ for all $\ell\in L$. Then for any closed normal subgroup $N\unlhd G$ one has $\image(\rho_{\ell|N})\in\Sigma_\ell(c)$ for all $\ell\in L$.

In particular, if $H\unlhd G$ is an open subgroup, then the above applies to any normal open subgroup $N\unlhd G$ that is contained in~$H$.
\end{coro}

\section{Fundamental groups: finiteness properties and ramification}
\label{FundGps}

The purpose of this section is to recall some finiteness properties of fundamental groups and to provide some basic results on ramification. Regarding the latter we draw from results by Kerz-Schmidt and Wiesend (cf.~\cite{kerzschmidt}) and from de Jong on alterations (cf.~\cite{dejong}).

We begin with a finiteness result of which a key part  is from \cite{gp}.
\begin{prop} \label{kl}
Suppose that either $k$ is a finite field and $S$ is a smooth proper $k$-variety or that $k$ is a number field and $S$ is a smooth $k$-variety, and denote by $K=k(S)$ the function field of $S$. For $d\in\Nn$, let $\mathcal{M}_d$ be the set of all finite Galois extensions $E/K$ inside $\widetilde{K}$ such that $\Gal(E/K)$ satisfies $\Jor(d)$ and such that $E$ is unramified along $S$. Then there exists a finite Galois extension $K'/K$ which is unramified along $S$ such that $E\subset \widetilde{k}K'$ for every $E\in \mathcal{M}_d$. 
\end{prop}

\begin{proof} Let $\Omega=\prod_{E\in\mathcal{M}_d} E$ be the compositum of all fields in $\mathcal{M}_d$. 
For every $E\in \mathcal{M}_d$ the group $\Gal(E/K)$ satisfies $\Jor(d)$ and hence 
there is a finite Galois extension $E'/K$ inside $E$ such that 
$[E':K]\le d$ and such that $E/E'$ is abelian. Define $$\Omega'=\prod_{E\in \mathcal{M}_d} E'.$$ 
Then $\Omega/\Omega'$ is abelian. Let $k_0$ (resp. $\kappa'$, resp. $\kappa$) 
be the algebraic closure of $k$ in $K$ (resp. in
$\Omega'$, resp. in $\Omega$). 
$$\xymatrix{&K\ar@{-}[r]\ar@{-}[d] & \Omega'\ar@{-}[d]\ar@{-}[r] & \Omega\ar@{-}[d]\\
k\ar@{-}[r] & k_0\ar@{-}[r] & \kappa'\ar@{-}[r] & \kappa\\ }$$

It suffices to prove the following 

{\bf Claim.} The extension $\Omega/\kappa K$ is finite.

In fact, once this is shown, it follows that the finite separable extension $\Omega/\kappa K$ has a primitive element
$\omega$. Then $\Omega=\kappa K(\omega)$ and $K(\omega)/K$ is a finite separable extension. Let $K'$ be the
normal closure of $K(\omega)/K$ in $\Omega$. Then $\tilde{k}K'\supset \kappa K'\supset \kappa K(\omega)=\Omega$ as
desired.

In the case $k=\Qq$  the claim has been shown in \cite[Proposition 2.2]{gp}. Assume from now on that $k$ is finite. It remains
to prove the claim in that case. The structure morphism $S\to\Spec(k)$ of the smooth scheme 
$S$ factors through $\Spec(k_0)$ and
$S$ is a geometrically connected $k_0$-variety. The profinite 
group $\pi_1(S\times_{k_0} \Spec({\widetilde{k}}))$ is topologically finitely generated (cf.~\cite[Thm.~X.2.9]{SGA1}) and
$\Gal(k_0)\cong \hat{\Zz}$. Thus it follows by the exact sequence (cf.~\cite[Thm.~IX.6.1]{SGA1})
$$1\to \pi_1(S\times_{k_0} \Spec({\widetilde{k}}))\to \pi_1(S)\to \Gal(k_0) {\to 1}$$
 that $\pi_1(S)$ is topologically 
finitely generated. Thus there are
only finitely many extensions of $K$ in $\tilde{K}$ of degree $\le d$ which are unramified along $S$. 
It follows that $\Omega'/K$ is a {\em finite} extension. Thus $\kappa'$ is a finite
field. If
we denote by $S'$ the
normalization of $S$ in $\Omega'$, then $S'\to S$ is finite and \'etale, hence $S'$ is a smooth proper geometrically connected $\kappa'$-variety. 
Furthermore $\Omega/\Omega'$ is abelian and unramified along $S'$.
Hence $\Omega/\kappa \Omega'$ is finite by Katz-Lang (cf.~\cite[Thm.~2, p. 306]{katzlang}). As $\Omega'/K$ is finite, it follows
that $\Omega/\kappa K$ is finite. 
\end{proof}

Our next aim is to introduce several notions of ramification, that are refinements of \cite{kerzschmidt}, useful for coverings of general schemes. Let $E/K$ be a separable algebraic extension of fields. Let $v$ be a discrete rank $1$ valuation of $K$ and
$w$ an extension of $v$ to $E$. Let $\ell$ be a prime number different from $\mathrm{char}(k(v))$. The extension $E/K$ is said to be {\em tame} (resp. {\em $\ell$-tame}) at $w$ if the residue field extension
$k(w)/k(v)$ is separable and for every finite extension $F$ of $K$ inside $E/K$ the ramification index $[w(F^\times):v(K^\times)]$ is prime to $\mathrm{char}(k(v))$ (resp. is a power of $\ell$).
If $E/K$ is Galois and the ramification group $I_w$ is of order prime to $\Char(k(v))$ (resp. pro-$\ell$), then
the residue field extension $k(w)/k(v)$ is automatically separable and $E/K$ is tame (resp. $\ell$-tame) at $w$.
The extension $E/K$ is {\em tame} 
(resp. {\em $\ell$-tame})
at $v$ if the extension $E/K$ is tame (resp. $\ell$-tame) at every extension $w$ of $v$ to $E$.

We fix some terminology for curves. Let $k$ be a field of characteristic $p\ge 0$. A {\em curve} 
$C$ over $k$ is a smooth (but not necessarily projective) $k$-variety of dimension $1$. Denote by $P(C)$ the smooth projective model of the function field $k(C)$ (the model is unique up to isomorphism).~Then $P(C)$ contains $C$, and we set $\partial C:=P(C)\setminus C$.  Let $\ell$ be a prime number different from $p$. An
\'etale cover $C'\to C$ is called {\sl tame} (resp. {\sl $\ell$-tame}) if for any point $x\in\partial C$ with valuation $v_x$ of $k(C)$ and every connected component $Z'$ of $C'$, the extension $k(Z')/k(C)$ is tame (resp. $\ell$-tame)  at~$v_x$.

\begin{defi}\label{ramif-def} 
Let $S$ be a regular variety over a field $k$ of characteristic $p\ge0$. Let $f: T\to S$ be an \'etale cover.
Let $\ell$ be a prime different from~$p$.
\begin{enumerate}
\item The cover $f: T\to S$ is {\em curve tame} (resp. {\em curve $\ell$-tame}) if for all $k$-morphisms  $\phi\colon C\to S$ with $C$ a smooth curve over $k$, the base changed cover $f_C\colon C\times_ST\to C$ is tame (resp. $\ell$-tame).
\item Assume that $S$ is an open subscheme of a regular projective scheme $\overline S$ such that $D=\overline S\setminus S$ is a normal crossings divisor (NCD). Then $f: T\to S$ is {\em tame} (resp. {\em $\ell$-tame}) 
if for every discrete valuation $v$ of $K$ defined by a generic point of $D$ and every connected component $Z$ of $T$ the extension $k(Z)/k(S)$ is tame (resp. $\ell$-tame) at $v$.
\end{enumerate}

We extend the above notions to profinite \'etale covers by saying that such a cover is {\em pro-$\ell$ curve tame}
or {\em pro-$\ell$ tame} if these conditions hold for all subcovers of finite degree.
\end{defi}

\begin{defi}\label{ramif-def2}
Let $k$ be a field of characteristic $p\ge0$. Let $K/k$ be a finitely generated extension and $E/K$ an algebraic extension. 
Let $\ell$ be a prime different from~$p$.
\begin{enumerate}
\item The extension $E/K$ is {\em generically \'etale} if there exists a regular $k$-variety $S$ with function field
$K$ such
that for every finite extension $F$ of $K$ inside $E/K$ the normalization of $S$ in $F$ is \'etale over $S$.  
\item The extension $E/K$ is {\em divisor tame} (resp. {\em divisor $\ell$-tame}) if it is generically \'etale and
tame (resp. $\ell$-tame) at every discrete rank $1$ valuation $v$ of $K$ which is trivial on
$k$.\end{enumerate}
\end{defi}


\begin{Def}\label{tamelyramrep-def}
{\sl Suppose $k$ is a field of characteristic $p\ge 0$ and $K/k$ is a finitely generated extension. Let $\ell$ be a prime number different from $p$. 
We call a homomorphism $\Gal(K)\to G$ of profinite groups {\em $\ell$-tame over $k$} if 
the fixed field $E:=(K_s)^{\kernel(\rho)}$ is divisor $\ell$-tame over~$K$.}
\end{Def}

\begin{Rem} \label{divtame-rem} Let $k$ be a field of characteristic $p\ge 0$ and $S$ a regular $k$-variety. Let $f: T\to S$ be
a connected \'etale cover. Let $K=k(S)$ and $E=k(T)$. Let $\ell$ be a prime number different from $p$. 
In \cite[p.~12]{kerzschmidt} the cover $f\colon T\to S$ is defined to be {\sl divisor tame} if for any normal compactification $\overline S$ of $S$ and a point $s\in \overline S\setminus S$ with $\codim_{\ol S}s=1$, the extension $E/K$ is tame at the rank $1$ valuation $v_s$ on $K$. We claim that the cover $T/S$ is divisor tame in the sense 
of \cite[p.~12]{kerzschmidt} if and only if the extension $E/K$ is divisor tame in the sense of definition \ref{ramif-def2}.

Clearly our notion of divisor tameness implies that of \cite{kerzschmidt}. For the converse we follow closely the argument in \cite[Thm.~4.4]{kerzschmidt} proof of (ii)$\Rightarrow$(iii) though with different references. Let $w$ be a valuation of $E$ that is trivial on $k$ and denote by $v$ its restriction to $K$. Let $\overline S$ be a normal compactification of $S$, which exists by the theorem of Nagata \cite{Luetkebohmert}. By \cite[Prop.~6.4]{Vaquie}, there exists a blow-up $\overline S'$ of $\overline S$ with center outside $S$ such that $v$ is the valuation of a codimension $1$ point $s\in \overline S'$. By normalization, we may further assume that $\overline S'$ is normal. Both operations, blow-up and normalization, do not affect $S$, and so we may take for $\overline S$ a normal compactification of $S$ that contains a codimension $1$ point $s$ with valuation $v=v_s$. 
But then the divisor tameness of \cite{kerzschmidt} implies that $w/v$ is at most tamely ramified.
\end{Rem}
 
The following result is a variant of parts of \cite[Thm.~4.4]{kerzschmidt}:
\begin{Prop}\label{KS-variant}
Let $k,S,T,K,E,f, \ell$ be as in Remark~\ref{divtame-rem}. Then the following hold:
\begin{enumerate}\advance\itemsep by -.4em
\item The cover $T/S$ is curve-tame if and only if the extension $E/K$ is divisor tame.
\item Suppose $E/K$ is Galois. Then the cover $T/S$ is curve-$\ell$-tame if and only if the extension $E/K$ is divisor $\ell$-tame.
\item If $S$ is an open subscheme of a regular projective scheme $\overline S$ such that $D=\overline S\setminus S$ is a NCD, then both conditions from (b) are equivalent to $T/S$ being $\ell$-tame.
\end{enumerate}
The assertions (a)--(c) extend in an obvious manner to profinite covers.
\end{Prop}
\begin{proof} 
By Remark~\ref{divtame-rem}, part (a) of the proposition follows directly from the equivalence (i)$\Leftrightarrow$(ii) in \cite[Thm.~4.4]{kerzschmidt}.

For the proof of (b), suppose first that $T/S$ is curve-$\ell$-tame and assume that $E/K$ is not divisor $\ell$-tame. By (a) we know that $E/K$ is divisor tame. So let $w$ be a valuation of $E$ at which $E/K$ is not divisor $\ell$-tame. Denote by $K_1\subset K_2\subset E$ extensions of $K$ such that $E/K_1$ is totally ramified (and Galois) at $w$ and $K_2/K_1$ is of prime degree $\ell'\neq\ell$. As in Remark~\ref{divtame-rem}, there exists a normal compactification $\overline S$ of $S$ and a codimension $1$ point $s$ of $\overline S\setminus S$ that has a preimage $t$ in the normalization $\overline T$ of $\overline S$ in $E/K$ with $v_t=w$. We define $T_i,\overline T_i$, $i=1,2$, as the normalizations of $S$ or $\overline S$ in $K_i/K$, respectively. We claim that $T/T_1$ is curve-$\ell$-tame. 

To see this, observe first, that as with curve-tameness, it is a simple matter of drawing a suitable commutative diagram to see that curve-$\ell$-tameness is stable under base change. In particular the base change $T\times_ST_1\to T_1$ is curve-$\ell$-tame. However, considering the commutative fiber product diagram
$$\xymatrix@C+2pc{T\ar[d]\ar[dr]\ar@{-->}@/^/[r]^s&\ar[l] T\times_ST_1 \ar[d]\\
S&\ar[l] T_1\rlap{,}\\
}$$
we see that there is a canonical splitting $s\colon T\to T\times_ST_1$ over $T_1$. Hence $T$ is a connected component of $T\times_ST_1$ and as such the restriction $T\to T_1$ of the morphism $T\times_ST_1\to T_1$ inherits curve-$\ell$-tameness.

Having the claim at our disposal, the hypothesis $[K_2:K_1]=\ell'$ yields that for any curve $C_1$ mapping to $T_1$, the induced cover $C_1\times_{T_1}T_2\to C_1$ is everywhere unramified along $\partial C_1$. Now $\overline T_1$ is regular in codimension $1$, hence the regular locus $W_1$ contains $T_1$ as well as the divisor corresponding to $w|_{K_1}$. Let $W_2$ be its preimage in $\overline T_2$. Now by \cite[Prop.~4.1]{kerzschmidt}, which can be paraphrased as: {\em curve-unramifiedness implies unramifiedness over a regular base}, it follows that $W_2\to W_1$ is \'etale. But then $K_2/K_1$ is \'etale along $w$, a contradiction.

For the converse of (b) suppose that $E/K$ is divisor $\ell$-tame. We assume that there is a $k$-morphism $C\to S$ for $C$ a smooth curve such that $\pi\colon C\times_ST\to S$ is not $\ell$-tame along $\partial C$. Since $\Gal(E/K)$ acts faithfully on $C\times_ST\to C$, by passing to a subgroup and thus an intermediate extension of $E/K$ we may assume that $C\times_ST$ is irreducible. Since then $\Gal(E/K)$ is also the Galois group of the cover $\pi$, some further straightforward reductions allow us to assume that $[E:K]=\ell'\neq\ell$ for some prime $\ell'$ (which by (a) is different from $p$), and that $\Gal(E/K)\cong\BZ/\ell'$ is the inertia group above some valuation of $k(C)$. Following the argument in the proof of \cite[Thm.~4.4]{kerzschmidt} (v)$\Rightarrow$(i), we can find a discrete rank $d=\dim S$ valuation of $E$ that is ramified of order $\ell'$ (via a Parshin chain through the image of $\Spec k(C)$). But \cite[Lem.~3.5]{kerzschmidt} says that $E/K$ is ramified at a discrete rank $d$ valuation if and only if it is ramified at a discrete rank $1$ valuation. We reach a contradiction because by hypothesis $E/K$ is unramified at all discrete rank $1$ valuations.

Finally, we prove (c). It is clear that divisor $\ell$-tameness implies $\ell$-tameness. The proof that $\ell$-tameness implies curve $\ell$-tameness follows from the argument given in \cite[Prop.~4.2]{kerzschmidt}: there it is shown that tameness implies curve tameness. Consider a curve $C$ over $k$ and a morphism $\phi\colon C\to S$ over $k$. Then $\phi$ extends to a morphism $\ol\phi\colon P(C)\to \overline S$. Denote by $\ol{\phi(C)}$ the closure of $\phi(C)$ in $\overline S$. The ramification of $T\times_SC\to C$ occurs precisely at those points of $P(C)$ that under $\ol\phi$ map to $D\cap \ol{\phi(C)}$. To analyze the ramification, the proof of  \cite[Prop.~4.2]{kerzschmidt} appeals to Abhyankar's lemma. In the notation of {\sl loc.~cit.}, the ramification is then governed by indices $n_i$, $i=1,\ldots, r$, that are prime to $p$. By the $\ell$-tameness of $T\to S$, the $n_i$ must all be powers of~$\ell$. But then 
{\sl loc.~cit.} implies that $T\times_SC\to C$ is $\ell$-tame, and this completes the proof.
\end{proof}

Our formulation of divisor-tameness easily transfers under rather general field extensions:
\begin{Lem}\label{CodimOneLTameIsInherited}
Suppose that $\Char(k)=p>0$ and consider the following inclusions of fields:
$$\xymatrix{K \ar@{^{ (}->}[r] & K'\\
\ar@{_{ (}->}[u] k \ar@{^{ (}->}[r]& k'\ar@{_{ (}->}[u]\\ }$$
If $E/K$ is Galois and divisor $\ell$-tame over $k$, then so is $EK'/K'$ over~$k'$. 
\end{Lem}
\begin{proof}
It clearly suffices to prove the lemma in the case where $E/K$ is finite Galois. Then $E':=EK'$ is finite Galois over~$K'$. Let $w'$ be any discrete rank one valuation of $E'$ trivial on $k'$ and denote by $w$ its restriction to $E$, by $v'$ the restriction to $K'$ and by $v$ the restriction to $K$. We need to show that $[w'({E'}^*):v'({K'}^*)]$ is a power of $\ell$ and that the residue extension is separable. The latter can be taken care of at once: The extension $E/K$ is finite separable. Hence so is $E'/K'$, because a primitive element of $E/K$ will be such an element for $E'/K'$. For the same reason, separability is preserved by the extensions of completions $E'_{w'}/K'_{v'}$ at $v'$. Now by the Cohen structure theorem, the extension of residue fields is a subextension of $E'_{w'}/K'_{v'}$, and as such it must be separable. It remains to consider the index of the value groups.

Suppose first that $v(K^*)=0$. Then we must have $w(E^*)=0$, since otherwise, if  $\alpha\in E$ would satisfy $w(\alpha)\neq0$, then the sequence $(\alpha^n)_{n\in\BZ}$ would be linearly independent over $K$, a contradiction. This means that under the residue map of $E'$, the subfield $E$ is mapped injectively to the residue field of $E'$ at $w'$. But then $E/K$ defines purely a residue extension of $E'/K'$, and thus $w'({E'}^*)=v'({K'}^*)$.

Next assume that $w$ is non-trivial, so that by the above $v$ is non-trivial as well. We pass to the completions and note that $E_w/K_v$ and $E'_{w'}/K'_{v'}$ remain Galois extensions. By the Cohen structure theorem, $K_v$ now contains the residue field $k(v)$, and $E_v$ the residue field $k(w)$. In particular, $F=k(w) K_v$ is an unramified extension of $K_v$ and $E_w/F$ is totally ramified. We may thus consider these two cases separately. Suppose first that $E_w/K_v$ is unramified. Then $E'_{w'}=K'_{v'}k(w)$ where clearly $k(w)$ defines a separable extension of the residue field of $K'_{v'}$. Hence $E'_{w'}/K'_{v'}$ is unramified. We conclude $w'({E'}^*)=v'({K'}^*)$ which completes the argument.

Finally let $E_w/K_v$ be totally ramified. By our hypothesis, the extension $E/K$ is at most $\ell$-order ramified at $w$. It follows that $E_w/K_v$ is a Galois extension with $\Gal(E_w/K_v)$ an $\ell$-group. Now $\Gal(E'_{w'}/K'_{v'})$ injects into $\Gal(E_w/K_v)$ because $E'=KE$, and thus $[E'_w:K'_v]$ is a power of $\ell$. Since the order of  $w'({E'}^*)/v'({K'}^*)$ divides the degree $[E':K']$, the proof is complete.
\end{proof}

Combining ramification properties with finiteness properties of fundamental groups, we obtain the following criterion for a family of representations of $\Gal(K)$ with images in $\Jor_\ell(d)$, or with abelian images of bounded order, to become trivial over $\Gal(K'\wt k)$ for some finite $K'/K$. 

\begin{prop}\label{FiniteAndRamif-prop}
Let $k$ be a field and let $S/k$ be a normal $k$-variety with function field $K$. Let $L$ be a set of prime numbers which does not contain $\Char(k)$. Suppose $(\rho_\ell\colon \pi_1(S)\to G_\ell)_{\ell\in L}$ is a family of continuous homomorphisms such that if $\Char(k)>0$ each $\rho_\ell$ is $\ell$-tame. Under either of the following two conditions there exists a finite extension $K'$ of $K$ such that for all $\ell\in L$ we have $\rho_\ell(\Gal(K'\wt k))=\{1\}$.
\begin{enumerate}
\item The field $k$ is finite or a number field and there exists a constant $d\in\BN$ such that for each $\ell\in L$ the group $\image(\rho_\ell)$ lies in $\Jor_\ell(d)$. 
\item The field $k$ is algebraically closed and there exists a constant $c\in\BN$ such that for each $\ell\in L$ the group $\image(\rho_\ell)$ is of order at most~$c$.
\end{enumerate}
\end{prop}

\begin{proof} 
First we replace $K$ by a finite Galois extension and $S$ be the normalization in this extension, so that we can assume that $\rho_\ell(\pi_1(S))=\{1\}$ for all $\ell \le d$ or $\ell\le c$, respectively. Next we apply the result of de Jong on alterations (cf.~\cite[Thm.~4.1, 4.2]{dejong}). It provides us with a finite extension $k'$ of $k$, a smooth projective geometrically connected $k'$-variety $T'$, a non-empty open subvariety $S'$ of $T'$ and an alteration $f: S'\to S$, such that furthermore $D':=T'\setminus S'$ is a normal crossing divisor. We define $K'$ to be the function field of~$S'$, so that $K'/K$ is finite. (If $k$ is perfect, we could also assume that $K'/K$ is separable.)

Next observe, that if $\Char(k)$ is positive, then the (divisor) $\ell$-tameness of $\rho_\ell$ implies the (divisor) $\ell$-tameness of $\rho_\ell|_{\pi_1(S')}$ by Lemma~\ref{CodimOneLTameIsInherited}, and thus, by Proposition~\ref{KS-variant}, for each $\ell$ the extension $K'_\ell=(K'_s)^{\kernel(\rho_\ell|_{\pi_1(S')})}$ of $K'$ is $\ell$-tame. Because of the first reduction step in the previous paragraph, this implies that each $\rho_\ell|_{\pi_1(S')}$ is unramified at the generic points of $D'$. Purity of the branch locus (cf.~\cite[X.3.1]{SGA1}) now implies that all $\rho_\ell|_{\pi_1(S')}$ factors via $\pi_1(T')$. 

Now, in case (a), the assertion of the proposition follows from Proposition~\ref{kl}. In case (b) we use that, since $k$ is algebraically closed, the fundamental group $\pi_1(T')$ is topologically finitely generated (cf.~\cite[Thm.~X.2.9]{SGA1}), and that furthermore, if $\Char(k)=0$, the same holds true for $\pi_1(S')$ (cf.~\cite[II.2.3.1]{SGA7}). Hence there are only
finitely many possibilities for the fields $K_\ell'$ and $\prod_\ell K_\ell'$ is a finite extension of $K'$.
This completes the proof in case~(b).
\end{proof}

\section{An independence criterion}
\label{indep-sec}

From now on, let $k$ be any field, let $K/k$ be a finitely generated field extension and let $L$ be a set of prime numbers not containing $p := \Char(k)$. For every $\ell\in L$ let $G_\ell$ be a profinite group and let $\rho_\ell: \Gal(K)\to G_\ell$ be a continuous homomorphism. 

If for all $\ell\in L$ the groups $\image(\rho_\ell)$ are $n$-bounded at $\ell$, then by Theorem~\ref{lp-main} we have a short exact sequence $1\to M_\ell\to \image(\rho_\ell)\to H_\ell\to 1$ with $H_\ell\in \Jor_\ell(d)$ for $d=J'(n)$ and $M_\ell\in\Sigma_\ell(2^n)$. At the end of the previous section we have seen that a combination of tameness of ramification and results on fundamental groups allow one to control the $H_\ell$ in a uniform manner. In this section we shall show how to control $M_\ell$ in a uniform manner, if one has a uniform control on ramification. We begin by introducing the necessary concepts and then give the result.
\medskip

To $(\rho_\ell)_{\ell\in L}$ we attach the family $(\wt\rho_\ell)_{\ell\in L}$ by defining each $\wt\rho_\ell$ as the composite homomorphism 
$$\wt\rho_\ell\colon\Gal({K}) \buildrel\rho_\ell\over\longrightarrow \image(\rho_\ell)\to
\image(\rho_\ell)/Q_\ell$$ 
where $Q_\ell$ is the maximal normal pro-$\ell$ subgroup of $\im(\rho_\ell)$. Note that if $\rho_\ell$ is an $\ell$-adic representation, then $\wt\rho_\ell$ is essentially the semisimplification of the mod $\ell$ reduction of $\rho_\ell$.

\begin{Def}\label{rsconditions}
{\sl \begin{enumerate}
\item The family $(\rho_\ell)_{\ell\in L}$ is said to satisfy the condition $\CR(k)$, 
if there exist a finite extension $k'$ of $k$, a finite extension $K'/Kk'$ and a smooth $k'$-variety $U'$ with function field $K'$ such that for every $\ell\in L$ the homomorphism $\wt\rho_\ell \resg{\Gal(K')}$ is unramified along~$U'$.

\item The family $(\rho_\ell)_{\ell\in L}$ is said to satisfy the condition 
$\CS(k)$, if it satisfies $\CR(k)$ and if one can choose the field $K'$ for $\CR(k)$ such that each $\wt\rho_\ell|_{\Gal(K')}$ is $\ell$-tame.
\end{enumerate}}
\end{Def}

The condition $\CR(k)$ says that each member $\wt\rho_\ell$ is up to pro-$\ell$ ramification potentially generically \'etale in a uniform way. The condition $\CS(k)$ is a kind of semistability condition.  

\begin{exam} \label{abvarex}
{\em 
Set $L=\Ll\smallsetminus \{\Char(k)\}$ and let $A/K$ be an abelian variety. For every $\ell\in L$ denote by $\sigma_{\ell}\colon \Gal(K)\to \Aut_{\Zz_\ell}(T_\ell(A))$ the representation of $\Gal(K)$ on the $\ell$-adic Tate module $\invlim_{i\in\Nn} A[\ell^i]$. There exists a finite extension $k'$ of $k$ and a finite separable extension $K'/k'K$ such that $K'$ is the function field of some smooth  $k'$-variety $V'$. By the spreading-out principles of \cite{EGAIV3} there exists a non-empty open subscheme $U'$ of $V'$ and an abelian scheme $\CA$ over $U'$ with generic fibre $A$. This implies (cf.~\cite[IX.2.2.9]{SGA7}) that $\sigma_\ell$ is unramified along $U'$ for every $\ell\in L$. Hence the family $(\sigma_\ell)_{\ell\in L}$ satisfies condition $\CR(k)$.

In order to obtain also $\CS(k)$, we choose an odd prime $\ell_0\in L$, and we require the field $K'$ above to be finite separable over $k'K(A[\ell_0])$. Now let $v'$ be any discrete valuation of $K'$ which is trivial on $k'$, and let $R_{v'}$ be the discrete valuation ring of $v'$. Let $\mathcal{N}_{v'}/\Spec(R_{v'})$ be the N\'eron model of $A$ over $R_{v'}$. The condition $K'\supset K(A[\ell_0])$ forces $\mathcal{N}_{v'}$ to be semistable (cf.~\cite[IX.4.7]{SGA7}). This in turn implies that $\sigma_{\ell}|I(v')$ is unipotent (and hence $\sigma_{\ell}(I(v'))$ is pro-$\ell$) for {\em every} $\ell\in L$ (cf.~\cite[IX.3.5]{SGA7}). It follows that the family $(\sigma_\ell)_{\ell\in L}$ satisfies condition $\CS(k)$.}
\end{exam}

The following is the main independence criterion of this section:
\begin{prop}\label{geometricLemm}
Let $k$ be an algebraically closed field and let $K/k$ be a finitely generated extension. 
Suppose that $(\rho_\ell)_{\ell\in L}$ is a family of representations of $\Gal(K)$ that satisfies the following conditions:
\begin{enumerate}
\item The family $(\rho_\ell)_{\ell\in L}$ satisfies $\CR(k)$ if $\Char(k)=0$ and $\CS(k)$ if $\Char(k)>0$.
\item There exists a constant $c\in\Nn$ such that for all $\ell\in L$ one has $\rho_\ell(\Gal(K))\in\Sigma_\ell(c)$. 
\end{enumerate}
Then there exists a finite abelian Galois extension $E/K$ with the following properties.
\begin{enumerate}
\item[(i)] For every $\ell\in L$ the group $\rho_\ell(\Gal(E))$ lies in $\Sigma_\ell(c)$ and is generated by its $\ell$-Sylow subgroups; if $\ell>c$ then the group $\rho_\ell(\Gal(E))$ is generated by the $\ell$-Sylow subgroups of $\rho_\ell(\Gal(K))$. 
\item[(ii)] The group $\Gal(E)$ is a characteristic subgroup of~$\Gal(K)$.
\item[(iii)] The restricted family $(\rho_\ell\resg{\Gal(E)})_{\ell\in L\setminus\{2,3\}}$ is independent and $(\rho_\ell)_{\ell\in L}$ is almost independent.
\end{enumerate}
\end{prop}

\begin{proof}
We can assume that $\rho_\ell$ is surjective for all $\ell\in L$. Denote by $G_\ell^+$ the normal subgroup of $G_\ell$ which is generated by the pro-$\ell$ Sylow subgroups of $G_\ell$. Then $\ol{G_\ell}:=G_\ell/G_\ell^+$ is a finite group of order prime to $\ell$. Denote by $\pi_{\ell}:G_{\ell}\to \ol{G_{\ell}}$ the natural projection. As $G_\ell$ lies in $\Sigma_\ell(c)$, so does its quotient $\ol G_\ell$ by Lemma~\ref{sigmalemm}(b). Now any group in $\Sigma_\ell(c)$ of order prime to $\ell$ is abelian of order at most~$c$, and thus the latter holds for~$\ol G_\ell$.
%
%
%

Let $K_\ell^+$ be the fixed field in $K_s$ of the kernel of the map $\pi_\ell\circ \rho_\ell$, so that $\Gal(K_\ell^+/K)\cong \ol{G_\ell}$. Then the compositum $E=\prod_{\ell\in L} K_\ell^+$ is an abelian extension of $K$ such that $\Gal(E/K)$ is annihilated by~$c!$. From Proposition~\ref{FiniteAndRamif-prop}(b), which uses hypothesis (a), we see that $E/K$ is finite. Assertion (ii) is now straightforward: By definition of the $K_\ell^+$ the subgroups $\Gal(K_\ell^+)$ of $\Gal(K)$ are characteristic and hence so is their intersection~$\Gal(E)$.

We turn to the proof of (i): For every $\ell\in L$, from (ii) the group $\rho_\ell(\Gal(E))$ is normal in $G_\ell$, and hence it lies in $\Sigma_\ell(c)$ by Lemma~\ref{sigmalemm}. By construction $\rho_\ell(\Gal(E))\subset \rho_\ell(\Gal(K_\ell^+))=G_\ell^+$ and $N_\ell:=G_\ell^+/\rho_\ell(\Gal(E))$ is abelian and annihilated by $c!$. We claim that (1) $N_\ell$ is an $\ell$-group, so that $N_\ell$ is trivial if $\ell>c$, and that (2) $\rho_\ell(\Gal(E))$ is generated by its pro-$\ell$ Sylow subgroups. We argue by contradiction and assume that (1) or (2) fails. 

If (2) fails, then $\rho_\ell(\Gal(E))$ has a finite simple quotient of order prime to $\ell$. Because $\rho_\ell(\Gal(E))$ lies in $\Sigma_\ell(c)$, this simple quotient has to be abelian of prime order $\ell'$ different from $\ell$. Again by (b), the Galois closure over $K$ of the fixed field of this $\ell'$-extension is a solvable extension. Denote by $F$ either this solvable extension if (2) fails, or the extension of $K_\ell^+$ whose Galois group is canonically isomorphic to $N_\ell$ if (1) fails. In either case $F/K$ is Galois and solvable, and we have a canonical surjection $\pi_\ell'\colon G_\ell\onto \Gal(F/K)$. Arguing as in the first paragraph, it follows that $I_\ell$ surjects onto $\Gal(F/K)$. By construction $\Gal(F/K_\ell^+)$ is not an $\ell$-group. It follows from the definition of $K_\ell^+$ that the normal subgroup $\pi_\ell'(P_\ell)\subset \Gal(F/K)$ is a proper subgroup of $\Gal(F/K_\ell^+)$. But then its fixed field is a proper extension of $K_\ell^+$ which is at once Galois and of a degree over $K$ that is prime to $\ell$. This contradicts the definition of $K_\ell^+$, and thus (1) and (2) hold. This in turn completes the proof of~(i).

We now prove (iii). Denote by $\Xi_\ell$ the class of those finite groups which are either a finite simple group of Lie type in characteristic $\ell$ or isomorphic to $\BZ/(\ell)$. The conditions in (i) imply that every simple quotient of $\rho_\ell(\Gal(E))$ lies $\Xi_\ell$. But now for any $\ell,\ell'\ge5$ such that $\ell\neq\ell'$ one has $\Xi_{\ell}\cap \Xi_{\ell'}=\emptyset$ (cf.~\cite[Thm.~5]{bible}, \cite{artin}, \cite{KLMS}). The first part of (iii) now follows from \cite[Lemme 2]{bible}. The second part follows from the first, the definition of almost independence and from \cite[Lemme 3]{bible}.
\end{proof}



\section{Reduction steps}
\label{Reductions}

Let $k$, $K$, $L$, $p$ and $(\rho_\ell)_{\ell\in L}$ be as at the beginning of Section~\ref{indep-sec}. In the previous two sections we have described ramification properties of $(\rho_\ell)_{\ell\in L}$ and properties of $(\rho_\ell(\Gal(K)))_{\ell\in L}$ that were essential to control in a uniform way the groups $H_\ell$ and $M_\ell$ that occur in $\rho_\ell(\Gal(K))$ as in Theorem~\ref{lp-main}. The aim of this section is to explain how these properties for a general pair $(K,k)$ in our target Theorem \ref{main}, can be reduced to a pair where $k$ is the prime field and $K$ is finitely generated over it. Moreover we shall explain how one can reduce the proof of our target theorem to the case where $X$ is a smooth and projective variety over~$K$.

\medskip

\begin{Lem}\label{FamilyRestriction}
Suppose we have a commutative diagram of fields
$$\xymatrix{K \ar@{^{ (}->} [r]&K'\\
k \ar@{^{ (}->} [r]\ar@{_{ (}->} [u]&k'\ar@{_{ (}->} [u]\\
}$$
such that $K'$ is finite over $Kk'$. Then the following properties hold true:
\begin{enumerate}\advance\itemsep by -.5em
\item[(i)] If $(\rho_\ell)_{\ell\in L}$ satisfies $\CR(k)$, then $(\rho_\ell \resg{\Gal({K'})})_{\ell\in L}$ satisfies $\CR(k')$.
\item[(ii)] If $\Char(k)>0$ and $(\rho_\ell)_{\ell\in L}$ satisfies $\CS(k)$, then $(\rho_\ell\resg{\Gal({K'})})_{\ell\in L}$ 
satisfies $\CS(k')$.
\item[(iii)] If there exists a constant $c\in\BN$ such that for all $\ell\in L$, the group $\rho_\ell(\Gal({K\wt k}))$ lies in $\Sigma_\ell(c)$, then there exists a finite Galois extension $E'/K'$ such that for all $\ell\in L$, the group $\rho_\ell(\Gal(E'\wt k'))$ lies in $\Sigma_\ell(c)$.
\end{enumerate}
\end{Lem}

\begin{proof}
Considering the diagram
$$\xymatrix{K \ar@{=}[r]&\ar@{^{ (}->} [r]K&k'K \ar@{^{ (}->} [r]&K'\\
k \ar@{^{ (}->} [r]\ar@{_{ (}->} [u]&k'\cap K\ar@{_{ (}->} [u]\ar@{^{ (}->} [r]&k'\ar@{=}[r]\ar@{_{ (}->} [u]&k'\rlap{,}\ar@{_{ (}->} [u]\\
}$$
it suffices to prove the lemma in the following three particular cases: (a) $K'=K$, (b) $Kk'=K'$ and $K\cap k'=k$ (base change), (c) $k=k'$ and $K'$ is finite over $K$. 

For the proof of (iii) note that in case (a) the group $\Gal({K'\wt k'})$ is a closed normal subgroup of  $\Gal({K\wt k})$ and in case (b) the canonical homomorphism $\Gal({K'\wt k'})\to\Gal({K\wt k})$ is surjective. In either case we take $E=K'$. Finally in case (c) we define $E$ as the Galois closure of $K'$ over $K$. Then in cases (a) and (c) assertion (iii) follows from Corollary~\ref{sigmal-cor} with $H=\Gal(E)$ and~$G=\Gal(K)$. Case (b) is obvious.

In all cases, the proof of (ii) is immediate from Lemma~\ref{CodimOneLTameIsInherited}, once (i) is proved. Therefore it remains to prove (i). We first consider case (a). By replacing $k$, $k'$ and $K$ several times by finite extensions, we can successively achieve the following, where in each step the previous property is preserved: First, using de Jong's result on alterations (cf.~\cite[Thm.~4.1, 4.2]{dejong}), there exists a smooth projective scheme $X/k'$ whose function field is $K$. Second, by the spreading out principle, there exists an affine scheme $U'$ over $k$ whose function field is $k'$ and a smooth projective $U'$-scheme $\CX$ whose function field is $K$. Third, by hypothesis $\CR(k)$ there exists a smooth $k$-scheme $U$ whose function field is $K$ such that all $\wt\rho_\ell$ factor via $\pi_1(U)$. By shrinking $U$ we may assume it to be affine. Also we choose an affine open subscheme $\CV$ of $\CX$. The corresponding coordinate rings we denote by $R$ and $\CR$, respectively. Both of these rings are finitely generated over $k$. Since the fraction field of both is $K$, by inverting a suitable element $g\neq0$ of $R$ we have $R[g^{-1}]\supset \CR$, and similarly we can find $0\neq f\in\CR.
 $ such that $\CR[f^{-1}]\supset R$. Inverting both elements shows that we can find an affine open subscheme $V$ of both $U$ and $\CV$. In particular, the function field of $V$ is $K$, the scheme $V$ is smooth over $k$ and over $U'$ and the representations $\rho_\ell$ all factor via $\pi_1(V)$. The following diagram displays  the situation:
$$\xymatrix{
&\ar@{_{ (}->} [dl]V\ar@{^{ (}->} [dr]\ar[ddr]&V'\ar@{-->}[ddr]\ar@{-->}[l]&&\\
U\ar[ddr]&
&\ar[d]\CX&X\ar[l]\ar[d]&\ar[l]\Spec K\ar[dl]\\
&&U'\ar[dl]&\ar[l]\Spec k'\ar[dll]\\
&\Spec k&\\
}$$
Define $V'$ as the base change $V\times_{U'} \Spec k'$, so that $V'\to \Spec k'$ is smooth and affine. Now if $W\to U$ is any \'etale Galois cover, then the base change $W\times_{U}V'$ is an \'etale Galois cover of $V'$. We deduce that $\wt\rho_\ell$ factors via $\pi_1(V')$ for all~$\ell$, and thus we have verified $\CR(k')$ for~$(\wt\rho_\ell)_{\ell\in L}$.

Case (b): This is a base change. Therefore if we replace $k$ by a finite extension, then $U/k$ becomes a smooth variety with function field a finite separable extension of $Kk'$. But then $U\times_{\Spec k}\Spec k'$ is smooth over $k'$ and its function field is a finite separable extension of $K'=Kk'$. From this (i) is immediate.

Case (c): To see $\CR(k)$ over $K'$, let $k'\supset k$ and $K''\supset K'k'$ be finite extensions such that there exists a smooth $k'$-scheme $U$ with function field $K''$ such that all $\wt\rho_\ell$ factor via $\pi_1(U)$. Let $U'$ be the normalization of $U$ in $K'K''$. Now choose $k''\supset k'$ and $K'''\supset K'K''$ finite such that there is a smooth $k''$-scheme $U''$ with function field $K''$ and a finite morphism to $U'$. Then $\CR(k')$ is verified by~$U''$. 
\end{proof}

The following is a standard lemma from algebraic geometry about models of schemes over finitely generated fields.
\begin{Lem}\label{descendfg}
Let $k$ be a field, $K/k$ be finitely generated and $X$ be an separated algebraic scheme over $K$. Then there exists an absolutely finitely generated field $K_0\subset K$ and a separated algebraic scheme $X_0$ over $K_0$ such that $kK_0=K$ and $X_0\otimes_{K_0}K=X$. If in addition $X/K$ is smooth and/or projective, 
then one can choose $X_0$ and $K_0$ in a way so that $X_0/K_0$ is smooth and/or projective.
\end{Lem}

\begin{proof} 
Let $\mathfrak{K}$ be the set of all finitely generated subfields of $K$. Then $K=\bigcup_{K'\in \mathfrak{K}} K'$ and $\Spec(K)=\plim_{K'\in \mathfrak{K}} \Spec(K')$. There exists $K'\in \mathfrak{K}$ and a separated algebraic $K'$-scheme $X'$ such that $X=X'_{K'}$ (cf.~\cite[8.8.2]{EGAIV3} and \cite[8.10.5(v)]{EGAIV3}). If $X/K$ is projective, then one can choose $K'$ and $X'$ in such a way that $X'/K'$ is projective (cf.~\cite[8.10.5(xiii)]{EGAIV3}). If $X/K$ is smooth, then $X'/K'$ is smooth. Furthermore there exist $x_1,\cdots, x_t\in K$ such that $K=k(x_1,\cdots x_t)$. Define $K_0:=K'(x_1,\cdots, x_t)$ and
$X_0:=X'_{K_0}$. Then $kK_0=K$, the field $K_0$ is finitely generated and $X_0$ has the desired properties.
\end{proof}

For a separated algebraic scheme $X$ over $K$ and any $\ell\in\Ll\smallsetminus \{\Char(k)\}$ we denote by $\rho_{\ell,X}$ the representation of $\Gal(K)$ on $\bigoplus_{q\ge0}\big(H^q_c(X_{\wt K},\BQ_\ell)\oplus H^q(X_{\wt K},\BQ_\ell)\big)$.
\begin{Cor}\label{FirstReduction} 
Let $p$ be a prime number or $p=0$. Let $L=\Ll\setminus \{p\}$.
Suppose that for all absolutely finitely generated fields $K_0$ with prime field $k_0$ of characteristic $p$
and for all schemes $X_0$ that are separated algebraic over $K_0$, the following conditions are true:
\begin{enumerate}
\item The family $(\rho_{\ell,X_0})_{\ell\in L}$ satisfies $\CR(k_0)$ if $p=0$ and $\CS(k_0)$ if $p>0$.
\item There exists a constant $c\in\Nn$ and a finite extension $E_0$ of $K_0$ such that for all $\ell\in L$ one has $\rho_{\ell,X_0}(\Gal({E_0\wt k_0}))\in\Sigma_\ell(c)$.
\end{enumerate}
Let $k,K,X$ be as in Theorem~\ref{main} and $\rho_\ell=\rho_{\ell, X}$ for $\ell\in L$. 
Then there exists 
a finite Galois extension $K'/K$ and a finite Galois extension $E/K$ with $K'\subset E$ 
such that the assertions (a) and (b) and conclusions (i)-(iii) 
of Proposition~\ref{geometricLemm} about $(\rho_\ell)_{\ell\in L}$ hold true if one replaces $K$ by $K'$ in them. In particular Theorem~\ref{main} holds.
\end{Cor}
\begin{proof}
Let $k$ be an algebraically closed field, let $K$ over $k$ be a finitely generated extension field and let $X$ be a separated algebraic scheme over $K$. By Lemma~\ref{descendfg}, we can find $K_0\subset K$ absolutely finitely generated and $X_0$ a separated algebraic scheme over $K_0$ such that $X=X_0\otimes_{K_0}K$, and moreover if $X$ is smooth and/or projective over $K$, then the same can be assumed for $X_0$ over~$K_0$. Next let $c$ be the constant and $E_0$ the field 
as guaranteed by our hypotheses. Now Lemma~\ref{FamilyRestriction} yields a finite Galois extension $K'/E_0K$ such that $(\rho_{\ell,X}|_{\Gal(K')})_{\ell\in L}$ satisfies $\CR(k)$ if $p=0$ and $\CS(k)$ if $p>0$ and in addition that for all $\ell\in L$, the image $\rho_\ell(\Gal(K'))$ lies in $\Sigma_\ell(c)$. We can assume $K'/K$ Galois after replacing $K'$ by a finite extension.
By Proposition~\ref{geometricLemm}, there exists a finite abelian Galois extension $E$ of $K'$ such that the assertions and conclusions of Proposition~\ref{geometricLemm} hold true if one replaces $K$ by $K'$ in them. Furthermore $E/K$ is Galois because
$\Gal(E)$ is a characteristic subgroup of the normal subgroup $\Gal(K')$ of $\Gal(K)$.
\end{proof}

We now come to the second reduction step.

\begin{Def} \label{sssdefi} 
{\sl For a representation $\rho_\ell\colon\Gal({K})\to\GL_n(\BQ_\ell)$ we denote by $\rho_\ell^\sss$ its {\em strict semisimplification}, i.e., the direct sum over the irreducible subquotients of $\rho_\ell$ where each isomorphism type occurs with multiplicity one.}
\end{Def}

Note that $\image(\rho_\ell^\sss)=\image(\rho_\ell^\ssi)$ where $\rho_\ell^\ssi$ denotes the usual semisimplification of $\rho_\ell$. 

\begin{Lem}\label{SemisimplifAndS(k)}
For every $\ell\in L$ let $\rho_\ell$ and 
$\rho_\ell'$ be representations $\Gal(K)\to \GL_n(\BQ_\ell)$. Suppose that one of the following two assertions is true:
\begin{enumerate}\advance\itemsep by -.5em
\item $\rho_\ell^\sss=(\rho_\ell')^\sss$ for all $\ell\in L$, or
\item $\rho_\ell$ is a direct summand of $\rho_\ell'$ for all $\ell\in L$.
\end{enumerate}
Then the following hold:
\begin{enumerate}\advance\itemsep by -.5em
\item[(i)] if the family $(\rho'_\ell)_{\ell\in L}$ satisfies $\CR(k)$ then so does $(\rho_\ell)_{\ell\in L}$;
\item[(ii)] if the family $(\rho'_\ell)_{\ell\in L}$ satisfies $\CS(k)$ then so does $(\rho_\ell)_{\ell\in L}$;
\item[(iii)] for any $\ell\in L$, if $\rho_\ell'(\Gal({K\wt k}))$ lies in $\Sigma_\ell(c)$ then so does $\rho_\ell(\Gal({K\wt k}))$.
\end{enumerate}
\end{Lem}
Note that condition (a) is symmetric, so that under (a) each of (i)--(iii) is an equivalence. 

\begin{proof}
The proof of (i)--(iii) under hypothesis (a) is an immediate consequence of the simple fact that the kernel of $\image(\rho_\ell)\to\image(\rho_\ell^\ssi)=\image(\rho_\ell^\sss)$ is a pro-$\ell$-group. Assertions (i) and (ii) under hypothesis (b) are trivial. For (iii) note that hypothesis (b) implies that $\rho_\ell'(\Gal({K\wt k}))$ is a closed normal subgroup of $\rho_\ell(\Gal({K\wt k}))$, and so we can apply Lemma~\ref{sigmalemm}.
\end{proof}

The following important result is taken from the Seminaire Bourbaki talk of Berthelot on de Jong's alteration technique (cf.~\cite[Thm.~6.3.2]{Berthelot})
\begin{Thm}\label{Berthelot}
Let $k$ be a field and $K$ be a finitely generated extension field. Let $X$ be a separated algebraic scheme over $K$. Then there exists a finite extension $k'/k$, a finite 
separable extension $K'/Kk'$ and a finite set of smooth projective varieties $\{Y_i\}_{i=1,\ldots,k}$ over $K'$ such that for all $\ell\in \Ll\smallsetminus \{\Char(k)\}$ the representation  
$(\rho_{\ell,X}\resg{\Gal({K'})})^\sss$ is a direct summand of $\big(\bigoplus_i\rho_{\ell,Y_i}\big)^\sss$.
\end{Thm}

Before giving the proof of Theorem~\ref{Berthelot}, let us state the following immediate consequence of Theorem~\ref{Berthelot}, Lemma~\ref{SemisimplifAndS(k)} and Corollary~\ref{FirstReduction}:
\begin{Cor}\label{SecondReduction} Let $p$ be a prime number or $p=0$. Let $L=\Ll\setminus \{p\}$.
Suppose that for all absolutely finitely generated fields $K_0$ with prime field $k_0$ of characteristic $p$
and for all smooth projective $K_0$-varieties $X_0/K_0$, the following conditions are true:
\begin{enumerate}
\item The family $(\rho_{\ell,X_0})_{\ell\in L}$ satisfies $\CR(k_0)$ if $p=0$ and $\CS(k_0)$ if $p>0$.
\item There exists a constant $c\in\Nn$ and a finite extension $E_0$ of $K_0$ such that for all $\ell\in \Ll\setminus\{p\}$ one has $\rho_{\ell,X_0}(\Gal({E_0\wt k_0}))\in\Sigma_\ell(c)$.
\end{enumerate}
Let $k$ be an algebraically closed field of characteristic $p$. Let $K/k$ be 
a finitely generated extension and $X/K$ a separated algebraic scheme. 
Then there exists 
a finite Galois extension $E/K$ 
such that the following holds true:
\begin{enumerate}
\item[(i)] The family $(\rho_{\ell, X})_{\ell\in L}$ satisfies $\CR(k)$ if $p=0$ and $\CS(k)$ if $p>0$.
\item[(ii)] For every $\ell\in L$ the group $\rho_{\ell, X}(\Gal(E))$ lies in $\Sigma_\ell(c)$ and is generated by its $\ell$-Sylow subgroups.
\item[(iv)] The restricted family $(\rho_{\ell, X}\resg{\Gal(E)})_{\ell\in L\setminus\{2,3\}}$ is independent and $(\rho_{\ell, X})_{\ell\in L}$ is almost independent.
\end{enumerate}
In particular Theorem~\ref{main} holds.
\end{Cor}

Note that the auxiliary field $K'$ from corollary 6.3 disappears in corollary 6.7 since we only record those assertions and conclusions from Proposition \ref{geometricLemm}, that
are important later, and these do not involve the field $K'$.

\begin{proof}[Proof of Theorem~\ref{Berthelot}]
For completeness we provide details of the proof in \cite{Berthelot}. For $?\in\{c,\emptyset\}$ we denote by $\rho_{\ell,X,?}$ the representation of $\Gal(K)$ on $\bigoplus_{q\ge0}\big(H^q_?(X_{\wt K},\BQ_\ell)\big)$. It suffices to prove the theorem separately for the families $(\rho_{\ell,X,?})_{\ell\in L}$. We also note that whenever it is convenient, we are allowed (by passing from $K$ to a finite extension) to assume that $X$ is geometrically reduced over $K$. This is so because $H^q_?(X_{\wt K},\BQ_\ell)\cong H^q_?(X_{\wt K,\red},\BQ_\ell)$ for any $q\in\BZ$ and $?\in \{\emptyset,c\}$. We first consider the case of cohomology with compact supports. The proof proceeds by induction on $\dim X$. 

We may assume, by passing from $K$ to a finite extension, that $X$ is geometrically reduced. Then $X_{\wt K}$ is generically smooth. After passing from $K$ to a finite extension, we can find a dense open subscheme $U\subset X$ that is smooth over $K$. By the long exact cohomology sequence with supports (cf.~\cite[Rem.~III.1.30]{milne}) we have for any $\ell$ an exact sequence 
$$\ldots\longto H^i_c(U_{\wt K},\BQ_\ell)\longto H_c^i(X_{\wt K},\BQ_\ell)\longto H^i_c((X\setminus U)_{\wt K},\BQ_\ell)\longto\ldots ,$$
so that for all $\ell$ the representation $(\rho_{\ell,X,c})^\sss$ is a direct summand of $(\rho_{\ell,U,c})^\sss\oplus(\rho_{\ell,X\setminus U,c})^\sss$. By induction hypothesis, it thus suffices to treat the case that $U$ is smooth over $K$. By the induction hypothesis it is also sufficient to replace $U$ by any smaller dense open subscheme, and it is clearly also sufficient to treat the case where $U$ is in addition geometrically irreducible.
 
By de Jong's theorem on alterations (cf.~\cite[Thms.~4.1, 4.2]{dejong}), after passing from $K$ to a finite extension, we can find a smooth projective scheme $Y,$ an open subscheme $U'$ of $Y$ and an alteration $\pi\colon U'\to U$. By replacing $K$ yet another time by a finite extension, we can assume that $U'\to U$ is generically finite \'etale. And now we pass to an open subscheme $V$ of $U$ and to $V':=\pi^{-1}(V)\subset U'$ such that $V'\to V$ is finite \'etale. By the induction hypothesis applied to $Y\setminus V'$ and again the long exact cohomology sequence for cohomology with support, we find that the assertion of the theorem holds true for $(\rho_{\ell,V',c})_{\ell\in L}^\sss$. From now on $\pi$ denotes the restriction to $V'$ and $\BQ_{\ell,X}$ will be the constant sheaf $\BQ_\ell$ on any scheme $X$. Since $\pi$ is finite \'etale, say of degree $d$, there exists a trace morphism $\Trace_{\pi}\colon \pi_* \BQ_{\ell,V'}\to\BQ_{\ell,V}$ whose composition with the canonical morphism $\BQ_{\ell,V}\to  \pi_* \BQ_{\ell,V'}$ is multiplication by~$d$ (cf.~\cite[Lem.~V.1.12]{milne}). In particular, the constant sheaf $\BQ_{\ell,V}$ is a direct summand of $\pi_* \BQ_{\ell,V'}$. Since $H_c^i(V'_{\wt K},\BQ_\ell)\cong H_c^i(V_{\wt K},\pi_*\BQ_\ell)$, we deduce that $(\rho_{\ell,V,c})^\sss$ is a direct summand of $(\rho_{\ell,V',c})^\sss$, and this completes the induction step.
 
Now we turn to the case $?=\emptyset$. The case when $X$ is smooth over $K$ but not necessarily projective is reduced, by Poincar\'e duality, to the case of compact supports: If $X$ is connected then one has $H^q(X_{\wt K},\BQ_\ell)\cong H^{2d-q}_c(X_{\wt K},\BQ_\ell(d))^\vee$ for $d=\dim X$ (cf.~\cite[Cor.~VI.11.12]{milne}), and one can reduce to the connected case by considering the connected components of $X$ seperately. 

Suppose now that $X$ is an arbitrary separated algebraic scheme over $K$. By what we said above, we may assume that $X$ is geometrically reduced. Again we perform an induction on $\dim X$. The first step is a reduction to the case where $X$ is irreducible, which maybe thought of as an induction by itself. Suppose $X=X_1\cup X_2$ where $X_1$ is an irreducible component of $X$ and $X_2$ is the closure of $X\setminus X_1$. Consider the canonical morphism $f\colon X_1\sqcup X_2\to X$. It yields a short exact sequence of sheaves 
\begin{equation}\label{sesforcomponents}
0\longto \BQ_{\ell,X}\longto f_*\BQ_{\ell,X_1\sqcup X_2}\longto \CF \longto 1
\end{equation}
where $\CF$ is a sheaf on $X$. Consider the inclusion $i\colon X_0\into X$  for $X_0:=X_1\cap X_2$.
We claim that $\CF\cong i_* \BQ_{\ell,X_0}$. To see this observe first that if we compute the pullback of the sequence along the open immersion $j\colon X\setminus X_0 \into X$, then $\CF$ vanishes and the morphism on the left becomes an isomorphism. In particular, $\CF$ is supported on $X_0$. To compute the pullback along the closed immersion $i$ we may apply proper base change, since $f$ is proper. But now the restriction of $f$ to $X_0$ is simply the trivial double cover $X_0 \sqcup X_0\onto X_0$, so that $i^*\CF\cong \BQ_{\ell,X_0}$. This proves the claim because $\CF\cong i_*i^*\CF$, as $\CF$ is supported on $X_0$. 

By an inductive application of the long exact cohomology sequences to sequences like~(\ref{sesforcomponents}), it suffices to prove the theorem for schemes $X$ that are geometrically integral and separated algebraic over $K$. In this case, the proof follows by resolving $X$ by a smooth hypercovering $X_\bullet$ see \cite[p.~51]{dejong}, \cite[6.2.5]{deligne} and the proof of \cite[Thm.~6.3.2]{Berthelot}. Since the hypercovering yields a spectral sequence that computes the cohomology of $X$ in terms of the cohomologies of the smooth $X_i$, and this for all $\ell$, and since only those $X_i$ with $i\le 2\dim X$, contribute to $X$, the induction step is complete since we have reduced the case of arbitrary $X$ to lower dimensions and to smooth~$X_i$.
\end{proof}

\section{Proof of the main theorem}\label{TheProof}
In this final section, we study the family $(\rho_{\ell,X})_{\ell\in L}$ in the particular case where $K$ is absolutely finitely generated with prime field $k$. We shall establish properties $\CR(k)$ and if $\Char(k)>0$ also property $\CS(k)$. This will use Deligne's purity results from Weil I (cf.~\cite[Thm.~1.6]{deligneweil1}) as well as the global Langlands correspondence for function field proved by Lafforgue (cf.~\cite{Lafforgue}). 

\medskip

Let $k$ be a finite field and $U$ a smooth $k$-variety. For every closed point $u\in U$ let $k(u)$ be the 
(finite) residue field of $u$, and let $D(u)\subset \pi_1(U)$ be the corresponding decomposition group (defined only
up to conjugation). 
Denote by $\mathrm{Fr}_u\in D(u)$ the preimage under the canonical isomorphism $D(u)\cong \Gal(k(u))$ of the map $\sigma_u: x\mapsto x^{|k(u)|}$, so that within $\pi_1(U)$, the automorphism $\mathrm{Fr}_u$ is also defined only up to conjugation.

\begin{Prop}\label{basechange}
Suppose $K$ is a finitely generated field of characteristic $p\ge 0$. Let $k$ be the prime field of $K$.
Let $X/K$ be  a smooth projective scheme. 
\begin{enumerate}
\item[a)] There exists a finite separable extension $K'/K$ and smooth $k$-variety 
$U'/k$ with function field $K'$ such that for each $q\ge 0$ and every $\ell\in\Ll\smallsetminus \{p\}$, 
the representation $\rho_{\ell,X}^{(q)}\resg{\Gal(K')}$ of
$\Gal(K')$ is unramified along $U'$.
\item[b)] If $p>0$, then the family of representations $(\rho_{\ell, X}^{(q)}\resg{\Gal(K')})_{\ell\in\Ll\smallsetminus \{p\}}$
is strictly compatible and pure of weight $q$, that is: 
For every closed point $u'\in U'$ the characteristic polynomial $p_{u'}(T)$ of 
$\rho_{\ell, X}^{(q)}(\mathrm{Fr}_{u'})$ has integral coefficients, 
is independent of $\ell\in \Ll\smallsetminus \{p\}$, and the reciprocal roots of $p_{u'}(T)$ all have absolute value
$|k(u')|^{q/2}$.
\end{enumerate}
\end{Prop}

\begin{proof} Note that $k$ is perfect. There exists a finite separable extension $E/K$ such 
that $X_E$ splits up into a disjoint union of geometrically connected smooth projective $E$-varieties.
Thus, after replacing $K$ by a finite separable extension,
we can assume that $X/K$ is geometrically connected.
Let $(S_1,\cdots, S_r)$ be a separating transcendence basis of $K/k$. Identify $k(S_1,\cdots, S_r)$ with the
function field of $\Aa_r$ and let $S$ be the normalization of $\Aa_r$ in $K$. Then $S$ is a normal $k$-variety with
function field $K$. There exists an alteration $S'\to S$ such that $S'/k$ is a smooth variety 
and the function field $K'$ of $S'$ is
a finite separable extension of $K$ (cf.~\cite[Thm.~4.1, Remark 4.2]{dejong}).
Let $\mathfrak{A}$ be the set of all affine open subschemes of $S'$. Then $\Spec(K')=\plim_{U'\in 
\mathfrak{A}} U'$. There exists $U'\in \mathfrak{A}$ and a projective $U'$-scheme $f: \mathcal{X}'\to U'$ such that $\mathcal{X}'\times_{U'} \Spec({K'})=X_{K'}$. 
(cf.~\cite[8.8.2]{EGAIV3} and \cite[8.10.5(v) and (xiii)]{EGAIV3}). By the theorem of generic smoothness,
after shrinking $U'$ and $\mathcal{X'}$, we can assume
that $f$ is smooth. Furthermore, after removing the support of $f_*(\OO_{\mathcal{X}'})$ from $U'$ and shrinking $\mathcal{X}'$ 
accordingly, we can assume that $f$ has geometrically connected fibres.

Let $q\ge 0$ and $u'\in U'$. Define $k':=k(u')$ and $X_{u'}:=\mathcal{X}'\times_{U'} \Spec(k')$. 
Then for every $\ell\in\Ll\smallsetminus\{p\}$ the \'etale sheaf
$R^q f_* \Zz_\ell$ is lisse and compatible with any base change (cf.~\cite[VI.2., VI.4]{milne}).
Thus $\rho_{\ell, X}^{(q)}\resg{\Gal(K')}$ factors through $\pi_1(U')$ and  a) holds true. 
Furthermore it follows that $H^q(X_{\wt K}, \Qq_\ell)$ can be 
identified with $H^q(X_{u', \wt{k'}}, \Qq_\ell)$ in a way compatible with the Galois actions. Assume $p>0$. 
Part b) then follows applying
Deligne's theorem on the Weil conjectures (cf.~\cite[Thm.~1.6]{deligneweil1}) to~$X_{u'}$. 
\end{proof}

\begin{Lem}
Let the notation be as in the previous proposition and suppose we are in the situation of part (b), so that $k$ is a finite field. Fix $q\in\BZ$ and denote by $n$ the dimension of $H^q(X_{\wt K},\BQ_\ell)$ which is independent of $\ell\neq p$. Then the following hold:
\begin{enumerate}
\item For any smooth curve $C$ over $k$ and morphism $\phi\colon C\to U'$, there exist irreducible cuspidal automorphic representations $(\pi_{C,i})_{i=1,\ldots,m_\phi}$ of $\GL_{n_i}(\BA_{k(C)})$ such that $\sum_in_i=n$ and~for~all $\ell$ the representation $(\rho_{\ell, X}^{(q)}\circ\phi_*)^\ssi$, where $\phi_*\colon \pi_1(C)\to\pi_1(U'),$ agrees with the $\ell$-adic representation $\oplus_i \rho_{\ell,\pi_{C,i}}$ attached to $\pi_{C,i}$ via the global Langlands correspondence in \cite[p.~2]{Lafforgue}.
\item If for some prime $\ell_0\ge3$ there is a $\BZ_{\ell_0}$-lattice $\Lambda$ of the $\BQ_{\ell_0}$ representation space underlying $\rho_{\ell_0, X}^{(q)}$ that is stabilized by $\Gal({K'})$ and such that $\Gal({K'})$ acts trivially on 
$\Lambda/\ell_0 \Lambda$, then in (a) for all $\phi\colon C\to U'$ the representations $\pi_{C,i}$ are semistable\footnote{We call an automorphic representation $\pi$ of $\GL_n$ {\sl semistable} at a place $v$ if under the bijective local Langlands correspondence between local representations and Frobenius semi-simplified Weil-Deligne parameters, the Weil-Deligne parameter of $\pi_v$ is unramified when restricted to the Weil group. This definition is for instance used in \cite[\S~2]{Rajan}. In {\sl loc.~cit.}\ on page~8 it is also pointed out that by \cite{bz}, this notion of semistability is equivalent to that of the automorphic representation having an Iwahori fixed vector.} at all places of $P(C)\setminus C$. In particular, for all $\ell\neq p$, the representation $\rho^{(q)}_{\ell,X}$ satisfies $\CS(k)$.
\end{enumerate}
\end{Lem}
\begin{proof}
Let $C$ be a smooth curve over $k$ and $\phi\colon C\to U'$ a  morphism over $k$. By Proposition~\ref{basechange}, the family of representations $(\rho_{\ell, X}^{(q)}\circ\phi_*)^\ssi$, where ${\ell\neq p},$ is pure of weight $q$, semisimple and strictly compatible as a representation of $\pi_1(C)$. By the main theorem of \cite[p.~2]{Lafforgue}, each $(\rho_{\ell, X}^{(q)}\circ\phi_*)^\ssi$ being pure and semisimple gives rise to a list of irreducible cuspidal automorphic representations via the global Langlands correspondence.\footnote{Lafforgue in \cite{Lafforgue} describes only a correspondence between $\BQ_\ell$-sheaves of weight $0$, i.e., Galois representations to $\GL_n$ whose determinant has finite image, and automorphic representations with finite order central character. A general statement is given in \cite[\S~0 and 0.1]{Deligne}} By the strict compatibility and the bijectivity of the correspondence on simple objects, this list is the same for all $\ell$ (up to permutation). This proves~(a).

Let now $\ell_0$ be as in (b). For each $\phi\colon C\to U'$ as in (a) consider the representation 
$(\rho_{\ell_0, X}^{(q)}\circ\phi_*)^\ssi$ as an action on the lattice $\Lambda$, that is trivial modulo $\ell_0 \Lambda$. Any filtration of $\Lambda\otimes_{\BZ_{\ell_0}}\BQ_{\ell_0}$ preserved by the action of $\Gal(k(C))$ induces a filtration of $\Lambda$. 
Denote by $\Lambda_C$ the induced lattice for $(\rho_{\ell_0, X}^{(q)}\circ\phi_*)^\ssi$. Then it follows that the induced action of $\Gal(k(C))$ on $\Lambda_C/\ell_0 \Lambda_C$ is trivial. Since $\ell_0>2$, this implies that series of the logarithm converges on the image of $(\rho_{\ell_0, X}^{(q)}\circ\phi_*)^\ssi$. This image being pro-$\ell$, a standard argument (cf.~\cite[Cor.~4.2.2]{Tate}) shows that the Weil-Deligne representation at any place of $P(C)$ attached to 
$(\rho_{\ell_0, X}^{(q)}\circ\phi_*)^\ssi$ is unramified when restricted to the Weil group. But then by the compatibility of the global and local Langlands correspondence (cf.~\cite[Thme.~VII.3, Cor.~VII.5]{Lafforgue}), 
this means that all $\pi_{C,i}$ are semistable at the places in $\partial C$ (they are unramified at all other places). We now apply the same argument to all $\ell\neq p$ in reverse, to deduce that all ramification of 
$(\rho_{\ell, X}^{(q)}\circ\phi_*)^\ssi$ is $\ell$-tame for all $\ell \neq p$. The point simply is that in a family of Galois representations arising from a set of automorphic forms, the Weil-Deligne representation at a place of $P(C)$ is independent of~$\ell\neq p$.

Finally from Proposition~\ref{KS-variant}(b), which is a variant of a result of Kerz-Schmidt-Wiesend, we deduce that for all $\ell\neq p$, the representation $\rho^{(q)}_{\ell,X}$ is divisor $\ell$-tame. Combined with Proposition~\ref{basechange}(b), this establishes~$\CS(k)$.
\end{proof}

\begin{Rem}
We would like to point out the parallel between the proof of $\CS(k)$ in part (b) of the previous lemma and in Example~\ref{abvarex}. In the proof of (b) we selected a prime $\ell_0$ and enlarged $K$ so that its image would act trivially on $\Lambda/\ell_0 \Lambda$ via $\rho_{\ell_0}$ for a $\Gal(K)$-stable lattice $\Lambda$. Then we could use the uniformity provided by automorphic representations (after restricting the $\rho_\ell$ to any curve) to deduce from this that all ramification was semistable in a sense. 

In Example~\ref{abvarex} we selected a prime $\ell_0$ and enlarged $K$ so that it contains $K(A[\ell_0])$. This requirement is equivalent to asking that $\Gal(K)$ acts trivially on the quotient $T_{\ell_0}(A)/\ell_0T_{\ell_0}(A)$ for the lattice $T_{\ell_0}(A)$ from the Tate-module. Then we used the uniformity provided by the N\'eron model $\CN$ of the abelian variety $A$ over any discrete valuation ring in $K$ to deduce $\CS(k)$. The semistability of $\CN$ follows from the condition at the single prime $\ell_0$ -- in the analogous way that the semistability of automorphic representations was deduced from a single prime~$\ell_0$.
\end{Rem}

\begin{Cor}\label{CharPFamilyIsPSS}
Suppose $K$ is absolutely finitely generated and of characteristic $p>0$. Suppose $X/K$ is smooth projective. Then $(\rho_{\ell,X})_{\ell\in \BL\setminus\{p\}}$ satisfies~$\CS(k)$.
\end{Cor}
\begin{proof}
Fix a prime $\ell_0\ge3$. Choose a lattice $\Lambda$ underlying the representation $\rho_{\ell_0,X}$.
Replace $K$ by a finite extension such that $\Gal(K)$ acts trivially on $\Lambda/\ell_0 \Lambda$ via $\rho_{\ell_0,X}$. Apply now (b) of the previous lemma to deduce the corollary.
\end{proof}

\begin{Thm} \label{mainmain} Let $k'$ be a field of characteristic $p\ge 0$ 
and let $K'/k'$ be a finitely generated extension. Let $L=\Ll\smallsetminus \{p\}$.
Let $X'/K'$ be a separated algebraic scheme.
Then there exists a finite extension $E'/K'$ and a constant $c'\in\Nn$ with the following properties:
\begin{enumerate}
\item[(i)] For every $\ell\in L$ 
the group $\rho_{\ell, X'}(\Gal(\wt {k'}E'))$ lies in $\Sigma_\ell(c')$ and is generated by its $\ell$-Sylow subgroups.
\item[(ii)] The family $(\rho_{\ell, X'}\resg{\Gal(\wt{k'} E')})_{\ell\in L\setminus\{2,3\}}$ is independent and $(\rho_{\ell, X'}\resg{\Gal(\wt{k'} K)})_{\ell\in L}$ is almost independent.
\end{enumerate}
\end{Thm}

\begin{proof}
It suffices to establish conditions (a) and (b) of Corollary~\ref{SecondReduction}. For this, let $K$ be absolutely finitely generated, $k$ its prime field and $X/K$ a smooth projective variety. In this case we have proven that $\CR(k)$ holds in Proposition~\ref{basechange} and that $\CS(k)$ holds for $k$ of positive characteristic in Corollary~\ref{CharPFamilyIsPSS}. This verifies condition~(a) of Corollary~\ref{SecondReduction}.

For condition~(b) define $n$ to be the dimensions of $\rho_{\ell,X}$. Since $X/K$ is smooth projective, this dimension is independent of the chosen $\ell$. Thus all images $\rho_{\ell,X}(\Gal(K))$ are $n$-bounded at $\ell$. From Theorem~\ref{lp-main} we deduce for each $\ell$ the existence of a short exact sequence 
$$1\to M_\ell \to \rho_{\ell,X}(\Gal(K)) \to H_\ell \to 1$$
with $M_\ell$ in $\Sigma_\ell(2^n)$ and $H_\ell$ in $\Jor_\ell(J'(n))$, for the constant $J'(n)$ from Theorem~\ref{lp-main}. Consider the induced representations $\tau_\ell$ defined as the composite $\Gal(K)\stackrel{\rho_{\ell,X}}\to \rho_{\ell,X}(\Gal(K)) \to H_\ell$. Since $\rho_{\ell,X}$ satisfies $\CR(k)$ if $\Char(k)=0$ and $\CS(k)$ if $\Char(k)>0$, the hypotheses of Proposition~\ref{FiniteAndRamif-prop} are met if we replace $\Gal(K)$ by $\pi_1(U')$ with $U'$ from Proposition~\ref{basechange}. Hence there exists a finite extension $K'$ of $K$ such that $\tau_\ell(\Gal(K'\wt{k}))$ is trivial for all $\ell\neq p$. But this shows that $\rho_{\ell,X}(\Gal(K'\wt k))$ lies in $\Sigma_\ell(2^n)$ for all $\ell\neq p$, proving condition~(b) of Corollary~\ref{SecondReduction} and completing the proof.
\end{proof}

\begin{Rem}
Under the hypothesis of the above theorem, Corollary~\ref{SecondReduction} also tells us that $\CR(k')$ holds for $(\rho_{\ell,X'})_{\ell\in\BL}$ if $\Char(k')=0$ and that $\CS(k')$ holds for $(\rho_{\ell,X'})_{\ell\in\BL}$ if $\Char(k')=p>0$.
\end{Rem}

{\sc Gebhard B{\" o}ckle\\
Computational Arithmetic Geometry\\
IWR (Interdisciplinary Center for Scientific Computing)\\
University of Heidelberg\\
Im Neuenheimer Feld 368\\
69120 Heidelberg, Germany}\\
E-mail address: \texttt{\small boeckle@uni-hd.de} 
\par\medskip

{\sc Wojciech Gajda\\
Faculty of Mathematics and Computer Science\\
Adam Mickiewicz University\\
Umultowska 87\\
61614 Pozna\'{n}, Poland}\\
E-mail adress: \texttt{\small gajda@amu.edu.pl}
\par\medskip
\newpage

{\sc Sebastian Petersen\\ 
Institut INF1\\
Universit\"at der Bundeswehr\\
Werner-Heisenberg-Weg 39\\
85577 Neubiberg, Germany}\\
E-mail adress: \texttt{\small sebastian.petersen@unibw.de} 
\end{document}